\documentclass{article}
\usepackage{color}
\usepackage{srcltx}
\usepackage{fancyhdr}
\usepackage{latexsym}
\usepackage{amsfonts}
\usepackage{amscd}
\usepackage{epsfig}
\usepackage[cp850]{inputenc}
\usepackage[english]{babel}
\usepackage{ifthen}
\usepackage{float}
\usepackage{amsthm}
\usepackage{amsmath}
\usepackage{amssymb}
\usepackage{mathrsfs}
\usepackage[mathscr]{eucal}
\usepackage{dsfont}
\usepackage{amstext}
\usepackage{syntonly}

\usepackage{tikz} 
\usepackage{color}
\usepackage{graphicx}

  \usepackage{enumerate}
\theoremstyle{plain}
\newtheorem{theorem}{Theorem}
\newtheorem{proposition}[theorem]{Proposition}
\newtheorem{lemma}[theorem]{Lemma}
\newtheorem{corollary}[theorem]{Corollary}

\theoremstyle{remark}
\newtheorem{remark}[theorem]{Remark}

\theoremstyle{definition}
\newtheorem{definition}[theorem]{Definition}

\newtheorem{example}[theorem]{Example}

\newcommand{\precnprec}{\mathop{\prec \overset{n}{\cdots } \prec}}

 \newcommand{\tbigwedge}{\mathop{\textstyle \bigwedge }}
 \newcommand{\tbigvee}{\mathop{\textstyle \bigvee }}

\begin{document}

\title{A new approach on distributed systems: orderings and representability}






 \author{ Asier Estevan Muguerza}
\noindent {Universidad P\'ublica de Navarra, \\Dpto. EIM,  INARBE Institute\\
Campus Arrosad\'{\i}a. Pamplona, 31006, Spain.\\ asier.mugertza@unavarra.es} 


 \begin{abstract} 
In the present paper we propose a new approach on `distributed systems': the processes are represented through total  orders and the communications are characterized by means of biorders. The resulting distributed systems capture situations met in various fields (such as computer science, economics and decision theory). We investigate questions associated to the numerical representability of order structures, relating concepts of economics and computing to each other.  The concept of `quasi-finite partial orders' is introduced as a finite family of chains with a  communication between them. The representability of this kind of structure is studied, achieving a construction method for a  finite (continuous) Richter-Peleg  multi-utility representation.

  \end{abstract}
 
   \textbf{{Keywords:}}
   
Distributed systems,  partial orders, biorders, representability.

\maketitle


 \section{Introduction and motivation} \label{s1} 
 
 \setcounter{footnote}{0}
 \renewcommand{\thefootnote}{\arabic{footnote}}

In the present paper we focus on an ordered structure known as \emph{distributed system}. Although this concept belongs primarily to the field of computer science, its mathematical structure is common to many areas.

The representability issue appears too in a wide range of fields, such as  economics and decision making  
 \cite{Alc,ales,BRME,chateau,EvO,Luc,Peleg,Nish}, computing \cite{fid,dcomp,Lamport,virtual,ray,L1}, mathematical psychology \cite{Doig,biorderintro,Naka,valued}, etc. This interest on the representability of relations is usually due to  maximization problems (in economics and decision making),  the need to control non-linear processes that are being executed (computer science), the convenience of transform qualitative scales into quantitative ones (mathematical psychology), etc.

\medskip

Hence,
one of the goals of the present work is to bring to the computer field knowledge from other areas (order theory and decision theory in mathematics or  economics, for instance) that may be helpful when dealing with distributed systems.
For this purpose, first we redefine the concepts of a distributed system. For that,
 we shall use the concept of biorder. 
This study was started in \cite{dsjmp} for the particular case of distributed systems of two processes.   However, in order to formalize completely the concept of distributed system, a further study is needed, linking biorders between  $n$  totally ordered sets. 
 
 The idea of a biorder was studied  by Guttman (see \emph{Guttman Scales} in \cite{gut}) and by Riguet under the name of \emph{Ferrers relation} \cite{rig}. However, the concept   was introduced for first time by Andr\'e Ducamp and Jean-Claude Falmagne in 1969 in \cite{biorderintro}, and studied in depth  in 1984 by Jean-Paul Doignon, Andr\'e  Ducamp and Jean-Claude Falmagne in \cite{Doig}.  It is defined as follows:
 
 \emph{A biorder $\mathrel{\mathcal{R}}$ from $A$ to $X$ is a binary relation, with $\mathrel{\mathcal{R}}\subseteq A\times X$, satisfying that  for
every $a, b \in A$ and $x, y \in X$ 
$(a\mathrel{\mathcal{R}} x) \wedge(b\mathrel{\mathcal{R}} y) $ implies $ (a\mathrel{\mathcal{R}} y) \vee(b\mathrel{\mathcal{R}} x)$.}
 

\medskip

 The concept of biorder can also  be found just as a Ferrers relation \cite{ales,Doig}, that is, as a relation $\mathrel{\mathcal{R}}$ on a single set $X$ such that for any $x,y,z,t\in X$ it holds that $x\mathrel{\mathcal{R}} y$ with $ z\mathrel{\mathcal{R}} t$ implies that $x\mathrel{\mathcal{R}} t$ or $z\mathrel{\mathcal{R}} y$.
 Other kind of relations such as interval orders \cite{ales,Bio,subm,BRME,Fis2,f3,Sco,Tve} (in particular, the more restrictive  case of a semiorder \cite{gurea2,gurea1,gurea3,Luc}), may be considered particular cases of biorders \cite{ales,Doig}.

   In the present paper we focus on biorders defined between disjoint totally ordered sets (see Figure\ref{Fbi}), and we use them in order to redefine the concept of distributed system \cite{fid,Lamport,virtual,ray}.

 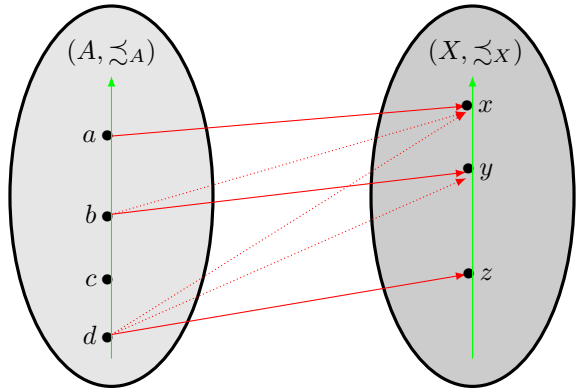
\begin{figure}[htbp]
\begin{center}
\begin{tikzpicture}[scale=0.8]
\begin{scope}[very thick]

\begin{scope}[very thick]

\draw[fill=gray!20!white] (-3,0) 
ellipse (48pt and 90pt);

\draw[fill=gray!40!white] (3,0) 
ellipse (48pt and 90pt);

\end{scope}

\end{scope}
\draw[red, -latex] (-3,-0.3) -- (2.9,0.4);
\draw[red, -latex] (-3,1) -- (2.9,1.5);
\draw[red, -latex] (-3,-2.3) -- (2.9,-1.3);

\draw[red, densely dotted, -latex] (-3,-2.3) -- (2.9,1.4);
\draw[red, densely dotted, -latex] (-3,-0.3) -- (2.9,1.4);
\draw[red,densely dotted, -latex] (-3,-2.3) -- (2.9,0.3);
    \draw (-2.8,-0.3) node[anchor=east] { $b\, \bullet$}; 
      \draw (-2.8,1) node[anchor=east] { $a\, \bullet$}; 
    \draw (-2.8,-1.4) node[anchor=east] { $c\, \bullet$};   
     \draw (-2.8,-2.3) node[anchor=east] { $d\,\bullet$};    
       \draw (3.5,0.4) node[anchor=east] { $\bullet \, y$};  
        \draw (3.5,1.5) node[anchor=east] { $ \bullet \, x$}; 
    \draw (3.5,-1.3) node[anchor=east] { $ \bullet \, z$};

\draw[green, -latex] (-3,-2.7) -- (-3,2);
\draw[green, -latex] (3,-2.7) -- (3,2);
\draw (-2.1,2.4) node[anchor=east] { $(A, \precsim_A)$} ;
\draw (4,2.4) node[anchor=east] { $(X, \precsim_X)$} ;

\end{tikzpicture}\\
\medskip
$ $

\caption{ 
A biordered pair of totally preodered sets.
}
\label{Fbi}
\end{center}\end{figure}

In a distributed system, distinct computers are connected to each other in order to achieve a common goal, this is known as \emph{distributed computing}.

These computers communicate with each other through messages that are sent and received. 
Each  computer  has  its own  internal  (physical)  clock,  so that it is possible to assing a number (a time) to each event of the proceess. Thus, from a mathematical point of view, each computer (or \emph{process}) is a totally ordered set (i.e., a chain) which can be represented through its local time.  
Hence, without a precise clock synchronization it is not possible to capture the causality relation between events of distinct processes. Moreover, if an event $b$ holds \emph{later}  (with respect to the time) than $a$, it does not imply that $a$ causally affects  $b$. On the contrary, if $a$ causally affects $b$, then $b$ must hold \emph{later} (with respect to an idyllic global time) than $a$. 
Finally, in these structures, the property called \emph{causal ordering of messages}  is  usually\footnote{Here we say \emph{usually} since it is a common property which is not implicit in Definition~\ref{Dds}.} satisfied:   if a computer $i$ sends two messages $m_1$ and then $m_2$ (so, such that $m_1$ has been sent before $m_2$) to the same computer $j$ ($i\neq j$), then  message $m_1$ must be received before  message $m_2$ (see Figure~\ref{Fcausalorderm}). \cite{Lamport,dcomp}

 \begin{figure}[htbp]
\begin{center} 
\begin{tikzpicture}[scale=1.2] 
\draw[fill=gray!20!white] (0,1) 
ellipse (20pt and 45pt);

\draw[fill=gray!20!white] (2,1) 
ellipse (20pt and 45pt);

    \draw[-latex] (0,0)--(0,2);
    \draw[-latex] (2,0)--(2,2);
        \draw[dashed, -latex] (0,2)--(2,0);
        \draw[dashed, -latex] (0,0)--(2,2);
\draw (0,2) node[anchor=east] {\small $a_2$};  
\draw (0,0) node[anchor=east] {\small $a_1$};  
\draw (2.45,2) node[anchor=east] {\small $b_2$};  
\draw (2.45,0) node[anchor=east] {\small $b_1$};  

\draw (0.17,2) node[anchor=east] {$\bullet$};  
\draw (0.17,0) node[anchor=east] {$\bullet$};   
\draw (2.2,2) node[anchor=east] {$\bullet$};  
\draw (2.2,0) node[anchor=east] {$\bullet$};  
\end{tikzpicture}\\
\medskip\vspace{0.4cm}
$ $
\caption{The vertical arrows represent the sequence of events of each process, that is, the direction of the time. The dashed arrows represent the sending of messages, from the sender to the receiver. Here, the causal ordering of messages is not satisfied.}
\label{Fcausalorderm}
\end{center}
\end{figure}
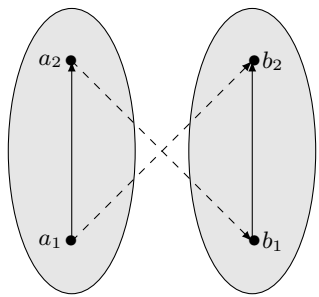

The concept of a \emph{distributed system} is usually defined as Lamport did  \cite{Lamport}:

\begin{definition}\label{Dds}\rm
An \emph{event} (illustrated by a point in Figure~\ref{Fbi} and Figure~\ref{FLam}) is a uniquely identified runtime instance of an atomic action of interest. It is an occurrence at a point in time, i.e., a happening at a cut of the timeline, which itself does not take any time. 
A \emph{process} (illustrated by a vertical line in Figure~\ref{Fbi} and Figure~\ref{FLam}) is a sequence of totally ordered events, thus, for any event $a$ and $b$ in a process, either $a$ comes before $b$ or $b$ comes before $a$. 
 A \emph{distributed system} consists of a collection of distinct processes which are spatially separated, 
  and which communicate with one another by exchanging messages (illustrated by red arrows in Figure~\ref{Fbi} and wavy arrows in Figure~\ref{FLam}). 
 It is assumed that sending or receiving a message is an event.
 \end{definition}

  Since each process consists of a sequence of events, each process  is a totally ordered set, 
  and the communication through messages between the processes will be defined by means of  biorders. 
 
 \begin{figure}[htbp]
 \begin{center}
\includegraphics[scale=0.85]{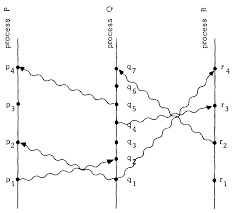}
 \caption{Illustration of a distributed system, taken from the  paper \cite{Lamport} of Leslie Lamport.}
\label{FLam}
\end{center}\end{figure}

 Moreover, this communication between  processes 
  defines a causal relation known as `causal precedence' or `happened before' relation \cite{fid,Lamport,virtual,ray}  (common in causality too \cite{Panan}, but now related to the theory of relativity, see also \cite{buse,penro}). 
 This causal relation 
  was defined by   Leslie Lamport in \cite{Lamport} as follows (see also \cite{dcomp}):

\begin{definition}\label{Dcp}\rm
 The  \emph{causal precedence} (denoted by $\rightarrow$) on the set of events of a distributed system is  the smallest relation satisfying the following three conditions:
 
\begin{enumerate}
\item If $a$ and $b$ are events in the same process, and $a$ comes before $b$, then $a\rightarrow b$.
\item If $a$ is the sending of a message by one process and $b$ is the receipt of the same message by another
process, then $a \rightarrow b$.
\item If $a \rightarrow b$ and $b \rightarrow c$ then $a \rightarrow c$.
\end{enumerate}
 
\noindent
 If $a \rightarrow b$, then it is said that $a$ \emph{causally precedes} $b$.
 
\end{definition} 
 
This definition was introduced by Leslie Lamport  in 1978 (see  \cite{Lamport}) and it has been used until nowadays. 
In the present paper we shall introduce a new definition of the concept through orderings.
Graphically, $ a\rightarrow b$ implies that there is  a  path of causality from event $ a$ to event $b$ (moving in the direction of the arrows, see Figure~\ref{Fbi} or Figure~\ref{FLam}), that is., $b$ is reachable from $ a$.

Two distinct events $a$ and $b$ are \emph{concurrent} if $a \nrightarrow b$ and $b \nrightarrow a$, that is, if they in no way can causally affect each other, so in that case it is not known which event happened first. It is assumed that $\rightarrow$ is  
irreflexive ($a \nrightarrow a$ for any event $a$), so, in case there are no cycles, $\rightarrow$ is a strict partial ordering on the set of all events in the system. 

\medskip

 Finally, we shall focus too on
finite Richter-Peleg multi-utility representations. The construction of these (continuous) representations for a given preorder may be a hard problem. 
 For this purpose, the study of the  (continuous, in case the sets are endowed with topologies) representability of biorders defined between totally preordered sets seems a right approach in order to achieve a  (continuous and finite) Richter-Peleg multi-utility representation of the corresponding causal precedence or happened before relation. 
 This idea consists in using (continuous) representations of distributed systems in order to represent (continuously) the corresponding causal precedence relation. 

  \medskip
 
 The structure of the paper goes as follows:
 
 After this introduction, a section of preliminaries is included.
 Next, in Section~\ref{sdefds} a new definition (and its motivation) of a distributed system is introduced, so that then the representability problem is studied in Section~\ref{srepds}, achieving an aggregation result for the case of line communications. In this section
 it is shown how to construct weak representations of distributed systems with line communications starting from pairs of functions that represent each biorder.
 Finally, in Section~\ref{snf} we focus on quasi-finite partial orders as an interesting family of partial orders. For this kind of orderings we also include a technique in order to construct a finite (and continuous) Richter-Peleg multi-utility.

 \section{Notation and preliminaries} \label{s2}

 From now on $A, B$ and $X$ as well as $X_1, ..., X_n$ will denote  nonempty (maybe infinite) sets.    When we speak of continuity of a real-valued mapping defined on a set $S$, we assume that some topology $\tau_S$ is given on $S$.

  \begin{definition}\label{firstdefi}\rm
  A \emph{binary relation} $\mathrel{\mathcal{R}}$ from $A$ to $X$ is a subset of the Cartesian product $A \times X$. In particular,  in the case that $A=X$, 
the binary relation $\mathrel{\mathcal{R}}$ is said to be defined  on $X$, and it is a subset of the Cartesian product $X \times X$. Given two elements $a\in A$ and $x\in X$, 
 we will use  notation $a \mathrel{\mathcal{R}} x$  
  to express that the pair $(a,y)$ 
   belongs to $\mathcal {R}$. 
 Associated to a binary relation $\mathrel{\mathcal{R}}$ from $A$ to $X$, 
  its \emph{negation} \rm 
   is the binary relation $\mathrel{\mathcal{R}}^c$ 
     from $A$ to $X$ 
   given by $(a,x) \in \mathrel{\mathcal{R}}^c \iff (a,x) \notin \mathrel{\mathcal{R}}$ for every $a\in A $ and $ x \in X$.

 
 Given two binary relations $\mathrel{\mathcal{R}}$ and $\mathrel{\mathcal{R}'}$ on a set $X$, it is said that $\mathrel{\mathcal{R}'}$ \emph{extends} or \emph{refines}  $\mathrel{\mathcal{R}}$ if 
 $x \mathrel{\mathcal{R}} y$ implies  $x \mathrel{\mathcal{R}'} y$, that is, if  $ \mathrel{\mathcal{R}} $ is contained in  $ \mathrel{\mathcal{R}'}.$
 
  The \emph{transitive closure} of a binary relation $\mathrel{\mathcal{R}}$ on a set $X$ is the transitive relation $\mathrel{\mathcal{R}}^+$ on set $X$ such that $\mathrel{\mathcal{R}}^+$ contains $\mathrel{\mathcal{R}}$ and $\mathrel{\mathcal{R}}^+$ is minimal.
 
The \emph{transitive reduction} of a binary relation $\mathrel{\mathcal{R}}$ on a set $X$ is, in case it exists,  the  smallest relation having the transitive closure of $\mathrel{\mathcal{R}}$ as its transitive closure.
 
Given a binary relation  $\mathrel{\mathcal{R}}$ on $X$, if two elements $x,y\in X$ cannot be compared, that is,    $\neg(x \mathrel{\mathcal{R}} y)$ as well as $\neg(y \mathrel{\mathcal{R}} x)$, then it is denoted by $ x \bowtie y $. We shall denote  $x\mathcal{I}y$ whenever $x\mathcal{R}y$ as well as  $y\mathcal{R}x$. 

 \end{definition} 
\smallskip
  
 Sometimes (depending on the ordering or when   different relations are mixed) the standard notation is different. We also include it here. 

 \begin{definition}\rm
  A \emph{preorder} $\precsim$ on $X$ is a binary relation on $X$ which is reflexive and transitive. 
 An antisymmetric preorder is said to be an \emph{order}. A \emph{total preorder} \rm $\precsim$ on a set $X$ is a preorder such that if $x,y \in X$ then $[x \precsim y] \vee [y \precsim x]$. A total order is also called a \emph{linear order}\rm, and a totally ordered set $(X,\precsim)$ is also said to be a \emph{chain}\rm.  Usually, an order that fails to be total is also said to be a  \emph{partial order} \rm and it is also denoted by $\preceq$. 
A subset $Y$ of a partially preordered set $(X, \precsim)$ is said to be an \emph{antichain} if $ x\bowtie y $ for any $x,y\in Y$.

If $\precsim$ is a preorder on $X$, then  the associated \emph{asymmetric} relation or \emph{strict preorder} is denoted by $\prec$ and the
associated \emph{equivalence} relation by $\sim$
and these are defined, respectively, by $[x \prec y \iff (x \precsim y) \wedge\neg(y \precsim x)]$ and $[x \sim y \iff (x\precsim y) \wedge (y \precsim x)]$. 
In the case of a finite partial order (also known as \emph{poset}), it is quite common to denote $\precsim$ by $\sqsubseteq$ and $\prec$ by $\sqsubset$, respectively.
The asymmetric part of a linear order (respectively, of a total preorder) is said to be a \emph{strict linear order}  (respectively, a \emph{strict total preorder}). 
 \end{definition}

\begin{definition}\label{Dnearcomplete}\rm
A preorder $\precsim$ on $X$ is said to be \emph{near-complete} if \mbox{$width(X, \precsim)$} $=n<\infty$. That is, if all antichains have cardinalities less or equal than $n$ (for some  $n\in \mathbb{N}$) as well as there is --at least-- one antichain which cardinality is $n$.
\end{definition}

\begin{definition}\rm
 A binary relation $\prec$ from $A$ to $X$ is 
a \emph{biorder} if it is Ferrers, that is, if for
every $a, b \in A$ and $x, y \in X$ the following condition holds:\\ 
$(a\prec x) \wedge(b\prec y) \Rightarrow   (a\prec y) \vee(b\prec x).$

 Related to $\prec$ we shall use the binary relation $\precsim$ from $X$ to $A$ given by $x\precsim a \iff \neg (a\prec x)$, $a\in A, \, x\in X$. It is also common to use $\succsim$ from $A$ to $X$ given by $a\succsim x \iff \neg (a\prec x)$, $a\in A, \, x\in X$.
\end{definition}

 \begin{definition} \label{tr}\rm
 Associated to a biorder
$\prec$ defined   from $A$ to $X$, we shall consider two new binary relations \cite{chateau,Doig}. 
 These binary relations are said to be the \emph{traces} \rm of $\prec$. They are defined on $A$ and $X$, respectively, and denoted by $\prec^*$,  $\prec^{**}$.  They are defined as follows: 

 First, $a \prec^* b \iff a \prec z \precsim b \ $
{\ for \ some \ } $ z \in X \ \ (a,b \in A),$ and, similarly,
$x \prec^{**} y \iff x \precsim c \prec y \ $
{\ for \ some \ }$ c \in A \ \ (x,y \in X).$ 
 \end{definition}
 
\begin{remark}\label{Rtraza}
 In the case of interval orders (so $A=X$) the binary relations denoted by $\prec^*$ and $\prec^{**}$ coincide with the ``\emph{left trace}" \rm  and ``\emph{right trace}" \rm  of the interval order. 
The names ``\emph{left trace}" \rm  and ``\emph{right trace}" have been used in the case of biorders too \cite{frontiers}, and other notations such as $\prec_A$ and $\prec_X$ or $\prec^l$ and $\prec^r$ can be found in literature \cite{frontiers,Doig,Naka}.  \end{remark}
\begin{remark} \label{rty}
We set $a \precsim^* b \iff \neg (b \prec^* a)$, \ \ $a \sim^*
b \iff a \precsim^* b \precsim^* a$, \ $x \precsim^{**}
y \iff \neg
(y \prec^{**} x)$ \ and \ $x \sim^{**}
y \iff x \precsim^{**} y \precsim^{**} x$.

These weak relations can be characterized as follows \cite{ales,Bio,Doig}:
\[ a\precsim^* b \iff \{b\prec x \Rightarrow  a\prec x\} , \text{ for any } x\in X. \]
\[ x\precsim^{**} y  \iff \{a\prec x \Rightarrow  a\prec y\} , \text{ for any } a\in A. \]
 
As a matter of fact, 
both the binary relations $\precsim^*$ and $\precsim^{**}$ are total preorders on $A$ and on $X$, respectively, if and only if the relation $\prec$ is a biorder \cite{Doig}. Hence, in that case the indifference relations $\sim^*$ and $\sim^{**}$ are in fact equivalence relations so, it is possible to define the quotient set $A\slash \sim^*$ and   $X\slash \sim^{**}$\cite{BRME,Doig}. 
\end{remark}

 Next Definition~\ref{lkeste} introduces the notion of representability\footnote{Other notions of representability appear for instance in  \cite{ales,Doig,gurea1,Naka}.} for total preorders and biorders. The goal of a representation is to convert a qualitative preference into a quantitative one, comparing real numbers instead of  elements of a set. 

 \begin{definition} \label{Dmulti}\rm

  Given a preorder $\precsim $ on $X$, a real function $u\colon X \to \mathbb{R}$ is said to be \emph{isotonic} or \emph{increasing} if for every $x ,y \in X$ the implication $x\precsim y \Rightarrow u(x)\leq u(y)$ holds true. In addition, if it also holds true that $x\prec y $ implies $u(x)< u(y)$, then $u$ is said to be a \emph{Richter-Peler utility representation}.
  
   A (not necessarily total) preorder $\precsim$ on a set $X$   is said to have a {\em
multi-utility representation} \cite{EvO}  if  there exists a family $\mathcal{U}$ of isotonic real functions such that 
for all points $x ,y \in X$  the equivalence
$\{x \precsim y \Leftrightarrow  \forall u \in {\mathcal{U}} \,\,(u(x) \leq u(y))\}$ 
 holds.   
 
 This kind of representation  always exists for every preorder $\precsim$ on $X$ (see Proposition~1 in  
 \cite{EvO}). It is also interesting to search for a {\em continuous multi-utility representation} of a preorder $\precsim$ when the set $X$ is endowed with a topology $\tau$  (cf., for instance, \cite{Alc,EvO}), as well as for multi-utility representations through a finite number of functions.

 When all the functions of the family $\mathcal{U}$  
are {\em order-preserving} for the preorder $\precsim$ (i.e., for all $u \in {\mathcal{U}}$, and $x,y \in X$, $x \prec y$ implies that $u(x) < u(y)$), then the representation is  called {\em Richter-Peleg multi-utility representation} 
\cite{Min2}. 
  \end{definition}

  In case of a poset, we shall use too the following concept\footnote{A similar term already exists in computer science under the name of \emph{random structures} (see \cite{L1}).}.

\begin{definition}\label{lakeste}\rm
 Let $(X,\sqsubseteq )$ be a finite partially ordered set with $|X|=n$. We shall say that a Richter-Peleg multi-utility representation  $\mathcal{U}$ is \emph{bijective} when each function $u\in \mathcal{U}$ is a bijection from $X$ to $\{1,...,n\}$.
 \end{definition}
\begin{remark}
The number of functions needed for a Richter-Peleg multi-utility coincides with the dimension of the partial order \cite{trotter}.
\end{remark}

 \begin{definition} \label{lkeste}\rm
 
  A total preorder $\precsim$ on $X$ is called \emph{representable} \rm if there is a real-valued function $u\colon X\to \mathbb R$ that is order-preserving, so that, for every $x, y \in X$, it holds that $[x \precsim y \iff u(x) \leq u(y)]$. The map $u$ is said to be an \emph{order-monomorphism} \rm (also known as a \emph{utility function} \rm for $\precsim$).

A biorder $\prec$ from $A$ to $X$  is
said to be \emph{representable} (as well as \emph{realizable with respect to $<$}) if there exist two real-valued functions $u \colon A \to \mathbb{R}$, $v \colon X \to \mathbb{R}$  such that $a\prec x \iff u(a)<v(x)$ ($a\in A$, $x\in X$). In this case it is also said
that the pair $(u,v)$ \emph{represents} $\prec$.
\medskip

Although we will work with this definition of representability for biorders  (realizable with respect to $<$), in Section~\ref{srepds} the following definition (introduced in  \cite{Doig}) is also needed:

A biorder $\prec$ from $A$ to $X$  is
said to be \emph{representable with respect to $\leq$} (as well as \emph{realizable with respect to $\leq$}) if there exist two real-valued functions $u \colon A \to \mathbb{R}$, $v \colon X \to \mathbb{R}$  such that $a\prec x \iff u(a)\leq v(x)$ ($a\in A$, $x\in X$). In this case we shall say 
that the pair $(u,v)$ represents $\prec$ \emph{with respect to $\leq$}.
\end{definition}

 \begin{definition} \label{conti} \rm  
Let $\prec$ be an asymmetric relation from a topological space $(A,\tau_A)$ to $(X,\tau_X)$. The relation $\prec$ is said to be \emph{$\tau_A$-continuous} \rm if the strict contour set 
$L_\prec (x) = \{ a \in A: a \prec x \}$ is a $\tau_A$-open set, for every $x \in X$. 
Dually, it is said to be \emph{$\tau_X$-continuous} \rm if the strict contour set $U_\prec (a) = \{ x \in X: a \prec x \}$ is a $\tau_X$-open set, for every $a \in A$.
We shall say that the relation is  \emph{bicontinuous} \rm if it is both $\tau_A$-continuous and $\tau_X$-continuous.

In particular, in the case of a single set $X$ (that is, $A=X$)  endowed with a single topology $\tau$ (so, $\tau=\tau_A=\tau_X$), the binary relation $\prec$  is said to be \emph{upper semi-continuous} if the strict contour set  $L_{\prec} (x) = \{ y \in X: y \prec x \}$ is $\tau$-open,  for every $x \in X$.  Dually, it is said to be 
\emph{lower semi-continuous} if the strict contour set
 $U_{\prec} (x) = \{ y \in X: x \prec y \}$  is $\tau$-open,  for every $x \in X$. 
 We shall say that the relation is  \emph{$\tau$-continuous} \rm if it is both upper and lower semi-continuous.

 \end{definition} 



\begin{definition}\rm
A biorder $\prec$ from $A$ to $X$  is
said to be \emph{continuously representable on $(A,\tau_A)$} \rm if it admits a  representation $(u,v)$ such that the  function $u\colon A \to \mathbb{R}$ is continuous when $A$ is given the topology $\tau_A$ and the real line is given its usual topology.
Dually, $\prec$ is said to be \emph{continuously representable on $(X,\tau_X)$} \rm if it admits a  representation $(u,v)$ such that the  function $v\colon X \to \mathbb{R}$ is continuous when $X$ is given the topology $\tau_X$ and the real line is given its usual topology. We say that the biorder is \emph{continuously representable} if it admits a  representation $(u,v)$ such that both  functions $u$ and $v$ are continuous.
\end{definition}

An example of a biorder (also related to the traces and the continuity of the representation) may be found  in \cite{dsjmp}. 
 %
%
%
%
  %
 %
 Let us recall now some characterizations of the representability, for total preorders and  biorders. 
 
 \begin{definition} \label{grande} \rm  
A total preorder $\precsim$ defined on $X$ is said to be \emph{perfectly separable} \rm if there exists a countable subset $D \subseteq X$ such that for every $x,y \in X$ with $x \prec y$ there exists $d \in D$ such that $x \precsim d \precsim y$. 

 Let $\prec$ be a biorder from $A$ to $X$. A subset $M$ of $A\cup X$ is said to be
 \emph{strictly dense}
   (see \cite{ales,Doig}) if for all $a\in A$ and $x\in X$, $a\prec x$ implies the existence of an element $m\in M$ such that  either 
 $m\in X \text{ and } a\prec m\precsim^{**} x$, or   $m\in A \text{ and } a\precsim^* m\prec x$.
\end{definition} 

\begin{remark}\label{Rcont}
Notice that if the number of classes defined by the equivalence relation $\sim^*$ (or $\sim^{**}$) of the trace $\precsim^*$ (resp. $\precsim^{**}$) on $A$ (resp. $X$) is countable, then the biorder has a trivial  strictly dense subset  made by means of the representatives of classes $A\slash \sim^*$ (resp. $X\slash \sim^{**}$).
\end{remark}

The following result is well-known \cite{BRME,Doig}. 

 \begin{theorem} \label{qwert} Let $A$ and $X$ be two nonempty sets.
 \begin{enumerate}[\em(a)]
 \item[\em(a)] A total preorder $\precsim$ on $X$ is representable if and only if it is perfectly separable.
  \item[\em(b)] A biorder $\prec$ from $A$ to $X$ is representable if and only if there exists a countable strictly dense subset.
\end{enumerate}
\end{theorem}

  A similar study on representability of interval orders, semiorders (see also \cite{gurea2}) and total preorders but now in the extended real line $\bar{\mathbb{R}}$  appears in \cite{gurea1}.

\medskip

 Given a biorder $\prec$ from  $A$ to $X$, it is possible to define the corresponding quotient sets as well as a new relation (see \cite{ales}) $\mathrel{\widehat{\prec}}$ from  $A\slash\sim^*$ to $X\slash\sim^{**}$ by
 $$\widehat{a}\mathrel{\widehat{\prec}} \widehat{x}\iff a\prec x,  \text{ for any } \widehat{a}\in A\slash\sim^*, \, \widehat{x}\in X\slash\sim^{**}.$$ 
 {  It is known (see \cite{ales}) that this relation $\mathrel{\widehat{\prec}}$ is well defined and that it is actually a biorder: the \emph{quotient biorder}.}
 It also holds true that $\widehat{a}\mathrel{\widehat{\prec}}^* \widehat{b}\iff a\prec^* b$,  and  $\widehat{x}\mathrel{\widehat{\prec}}^{**} \widehat{y}\iff x\prec^{**} y$,  for any  $\widehat{a}, \widehat{b}\in A\slash\sim^*$ and  for any $\widehat{x}, \widehat{y}\in X\slash\sim^{**}$.
 Thus,  
 any representation $(\widehat{u},\widehat{v})$ of $\mathrel{\widehat{\prec}}$ has the additional property that $\widehat{u}$ and  $\widehat{v}$ also represent the traces $\mathrel{\widehat{\prec}}^*$ and $\mathrel{\widehat{\prec}}^{**}$, respectively.

  Moreover, any representation $(\widehat{u},\widehat{v})$ of $\mathrel{\widehat{\prec}}$ delivers also a representation of $\prec$ just defining $u(a)=\widehat{u}(\widehat{a})$ and $v(x)=\widehat{v}(\widehat{x})$  (for any $a\in A$ and $x\in X$). Therefore, if $\mathrel{\widehat{\prec}}$ is representable then $\prec$ is representable, too. The converse is also true 
  as the following lemma shows \cite{ales}:
  
  \begin{lemma}\label{Lrepquotient}
  Let $\prec$ be a biorder from $A$ to $X$ and let $\mathrel{\widehat{\prec}} $ be the corresponding quotient biorder from  $A\slash\sim^*$ to $X\slash\sim^{**}$. Assume that $\prec$ admits a  representation $(u,v)$. Then, $\mathrel{\widehat{\prec}} $ is also representable.
  \end{lemma}

The following definition was introduced by Nakamura in \cite{Naka} and, as it is shown in  Corollary~\ref{reptrazas} (see also \cite{Naka} or Remark~1 in \cite{frontiers}), it is equivalent to the aforementioned order-denseness condition named `strictly dense'.

 \begin{definition} \label{Dnaka} \rm  
 Let $\prec$ be a biorder from $A$ to $X$. A pair of subsets $A^*\subseteq A$ and $X^*\subseteq X$ is said to be
 \emph{jointly dense} for $\prec$
  if for all $a\in A$ and $x\in X$, $a\prec x$ implies the existence of two elements $a^*\in A^*$ and $x^*\in X^*$ such that  
 $a\precsim^* a^*\prec x^*\precsim^{**} x$.
\end{definition}

Next corollary is well-known \cite{ales,frontiers,chateau,Naka}:

\begin{corollary}\label{reptrazas}
Let $\prec$ be a biorder from $A$ to $X$. The following statements are equivalent:
\begin{enumerate}
\item[($i$)] The biorder has a pair of jointly dense and countable  subsets.
\item[($ii$)] The biorder has a countable strictly dense subset.
\item[($iii$)] The biorder is representable.
\item[$(iv)$] The biorder is representable through a pair of functions $(u,v)$ with the additional condition that $u$ represents the trace $\precsim^*$ and $v$  the trace $\precsim^{**}$.
\end{enumerate}
\end{corollary}

\section{A new definition for Distributed System}\label{sdefds}

In the following pages, we  introduce a new definition for the concept of a distributed system with $n$ processes   (the original   definitions of these concepts --an event, a process and a distributed system-- have been introduced in Section~\ref{s1}).
But first,  we redefine the concepts of \emph{causal precedence} and \emph{communication}.

Since in the following pages many relations are going to appear, for the sake of clarity, from now we shall use the symbol $\mathcal{P}$ in order to refer to a  biorder relation, whereas we shall use the symbol $\precsim$ for total preorders (also for the traces associated to a biorder).

\begin{definition}\rm
Let $\{(X_k, \prec_k)\}_{k\in K}$ be a finite family of strict partially ordered  and disjoint sets (i.e., $X_i\cap X_j=\emptyset, \, \forall \, i\neq j$) and $\{\mathcal{P}_{ij}\}_{i\neq j}$ a family of relations from $X_i$ to $X_j$ (for any $i\neq j, \, i,j\in K $). The \emph{causal precedence} corresponding to  $\{\prec_k\}_{k\in K}\cup \{\mathcal{P}_{ij}\}_{i\neq j} $ on $X=\bigcup\limits_{k\in K} X_k$ is the transitive closure of the union $(\bigcup\limits_{k\in K} \prec_k) \bigcup (\bigcup\limits_{i\neq j}\{\mathcal{P}_{ij}\})$. This relation shall be denoted by $\rightarrow$:

$$ \{(\bigcup\limits_{k\in K} \prec_k) \bigcup (\bigcup\limits_{i\neq j}\{\mathcal{P}_{ij}\})\}^+= \rightarrow.$$ 
\end{definition}

\begin{remark}
Notice that, with the definition above, the absence of cycles is not guaranteed, as it is shown the Figure~\ref{Fcycle}. The existence of this cycles implies an error in the computing, known as \emph{deadlock} \cite{dcomp}. 
 
 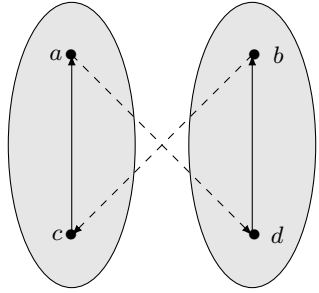
\begin{figure}[htbp]
\begin{center} 
\begin{tikzpicture}[scale=1.2] 
\draw[fill=gray!20!white] (0,1) 
ellipse (20pt and 45pt);

\draw[fill=gray!20!white] (2,1) 
ellipse (20pt and 45pt);

    \draw[-latex] (0,0)--(0,2);
    \draw[-latex] (2,0)--(2,2);
        \draw[dashed, -latex] (0,2)--(2,0);
        \draw[dashed, -latex] (2,2)--(0,0);
\draw (0,2) node[anchor=east] {\small $a$};  
\draw (0,0) node[anchor=east] {\small $c$};  
\draw (2.45,2) node[anchor=east] {\small $b$};  
\draw (2.45,0) node[anchor=east] {\small $d$};  

\draw (0.17,2) node[anchor=east] {$\bullet$};  
\draw (0.17,0) node[anchor=east] {$\bullet$};   
\draw (2.2,2) node[anchor=east] {$\bullet$};  
\draw (2.2,0) node[anchor=east] {$\bullet$};  
\end{tikzpicture}\\
\medskip
$ $
\caption{A cycle in a distributed system: a deadlock}
\label{Fcycle}
\end{center}
\end{figure}

If there is no cycle, then the transitive closure is a strict partial order. In the present paper we will assume that there is no error or  deadlock in the distributed system, that is, we shall assume that the causal precedence is a strict partial order.
\end{remark}

Now, recovering the idea of a distributed system in the spirit of Definition~\ref{Dds}, we first mathematically formalize the idea of \emph{communication} just as a finite relation between to distinct sets satisfying a `bijective' condition.

\begin{definition}\label{Dcomm}\rm
Let $A $ and $ B $ be two disjoint sets.
A \emph{communication from $A$ to $B$} is a finite binary relation $\mathcal{P}\subseteq A\times B$ (so, $|\mathcal{P}|<\infty$) such that for any $(a,b)\in \mathcal{P}$  and $a'\in A, \, b'\in B$ the following \emph{bijective} condition is satisfied:
 $$(a',b)\in \mathcal{P}\Rightarrow  a=a' \quad\text{ as well as }\quad (a,b')\in \mathcal{P}\Rightarrow  b'=b.$$

Here, the elements $a\in A$ such that $(a,b)\in \mathcal{P}$ (for some $b\in B$) are said to be the \emph{senders}, whereas  the elements $b\in B$ such that $(a,b)\in \mathcal{P}$ (for some $a\in A$) are said to be the \emph{receivers}\footnote{We call them senders/receivers in the spirit of Definition~\ref{Dds}.}.
\end{definition}

Now, we focus on communications between ordered sets. 
 
 \begin{definition}\rm
 Let $(A,\precsim_A)$ and $(B,\precsim_B)$ 
 be two disjoint  partially ordered sets. We shall say that a binary relation $\mathcal{P}$ is a \emph{causal biorder} from $A$ to  $B$ if 
  $\mathcal{P} \subseteq A\times B$ and for any $a,c\in A$, $b,d\in B$ it holds that
 $$(a\mathcal{P} b)  \text{ and } (c\mathcal{P} d)\quad \Rightarrow  \quad (a\rightarrow d)  \text{ or } (c\rightarrow b),$$
  where $\rightarrow$ denotes the corresponding causal precedence of $\precsim_A \cup \precsim_B \cup  \mathrel{ \mathcal{P}}$ on $A\cup B$.
 \end{definition}

 \begin{example}\rm
 Let $(A,\precsim_A)$ and $(B,\precsim_B)$ be two partially ordered sets as defined in Figure~\ref{Fcausalbi}. Let $\mathcal{P}$ and  $\mathcal{P}'$ be two relations from $A$ to $B$ defined by:
 $$ \mathcal{P}=\{(a_2, b_1), (a_3,b_3)\}  \quad \text{ and } \quad  \mathcal{P}'=\{ (a_2,b_2), (a_3,b_3)\}. $$
 It is straightforward to see that $\mathcal{P}$ is a causal biorder whereas $\mathcal{P}'$  is not.
  \end{example}
 
 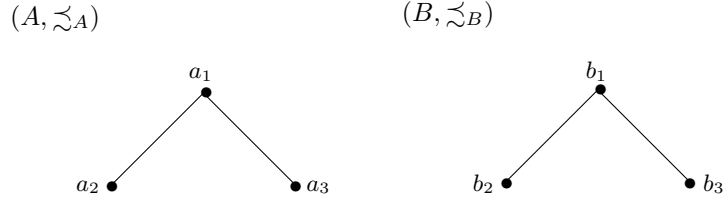
\begin{figure}[htbp]
\begin{center} 
\begin{tikzpicture}[scale=1.25] 
    \draw[] (0,2)--(1,3);
    \draw[] (2,2)--(1,3);
\draw (0,3.8) node[anchor=east] {$(A, \precsim_A)$}; 
\draw (0,2) node[anchor=east] {\small ${a}_2$};  
\draw (2.45,2) node[anchor=east] {\small ${a}_3$};  
\draw (1.2,3.2) node[anchor=east] {\small ${a}_1$};

\draw (0.2,2) node[anchor=east] { $\bullet$};   
\draw (2.15,2) node[anchor=east] { $\bullet$};  
\draw (1.2,3) node[anchor=east] {$\bullet$};
\end{tikzpicture}\qquad
\begin{tikzpicture}[scale=1.25] 
    \draw[] (0,2)--(1,3);
    \draw[] (2,2)--(1,3);
\draw (0,3.8) node[anchor=east] {$(B, \precsim_B)$}; 
\draw (0,2) node[anchor=east] {\small $b_2$};  
\draw (2.45,2) node[anchor=east] {\small $b_3$};  
\draw (1.2,3.2) node[anchor=east] {\small $b_1$};

\draw (0.2,2) node[anchor=east] { $\bullet$};   
\draw (2.15,2) node[anchor=east] { $\bullet$};  
\draw (1.2,3) node[anchor=east] {$\bullet$};
\end{tikzpicture}\\
\medskip
$ $
\caption{Two partially ordered sets, $(A,\precsim_A)$ and $(B,\precsim_B)$. }
\label{Fcausalbi}
\end{center}
\end{figure}

  \begin{definition} \rm
 Let $(A,\precsim_A)$ and $(B,\precsim_B)$ be two disjoint  partially ordered sets and
  $\mathcal{P}$  a communication from $A$ to $B$. 
We define the relation $\overline{\mathcal{P}}$  from $A$ to $B $ by 
 $$a\overline{\mathcal{P}}b \iff a\precsim_A a'\mathcal{P} b' \precsim_B b, \text{ for some } a'\in A, b'\in B.$$
That is,   $\overline{\mathcal{P}}=\precsim_{A}\circ \mathcal{P}\circ \precsim_{B}$.
 \end{definition}


 \begin{proposition}\label{Pbiorder}
 Let $(A,\precsim_A)$ and $(B,\precsim_B)$ 
 be two disjoint  partially ordered sets and $\mathcal{P}$ any relation from $A$ to $B$.
If $(A, \precsim_A)$ or $(B, \precsim_B)$ is a chain, then:

\begin{enumerate}
\item[$(i)$] $\mathcal{P}$ is a causal biorder.
\item[$(ii)$] $\overline{\mathcal{P}}$ is a biorder.
\end{enumerate} 
In particular, any communication  from $A$ to $B$ is a causal biorder.
 \end{proposition}

\begin{proof}
Let $a,x\in A$ and $b,y\in B$ be elements such that $a\mathcal{P} b$ and $x\mathcal{P} y$. If $(A, \precsim_A)$ is a chain, then $a\precsim_A x$ or $x\precsim_A a$ is satisfied. Hence, it holds that $a\precsim_A x \mathcal{P} y b$ or $x\precsim_A a \mathcal{P} b$, thus, $a \rightarrow b$ or $x\rightarrow b $. So, $\mathcal{P}$ is a causal biorder. We reason dually if $(B,\precsim_B) $ is a chain.

Let $a,b\in A$ and $x,y\in B$ be points  such that $a \overline{\mathcal{P}} x$ and $b \overline{\mathcal{P}} y$. Hence, by definition, there exist $a',b'\in A$ and $x',y'\in X$ such that $a\precsim_A a' \mathcal{P} x' \precsim_B x$ and  $b\precsim_A b' \mathcal{P} y' \precsim_B y$.
If  $(A, \precsim_{A})$ is a chain, we distinguish two cases:
\begin{enumerate}
\item If $a\precsim_A b$, then  $a\precsim_A b' \mathcal{P} y' \precsim_B y$, so $a\overline{\mathcal{P}} y$.

\item If $b\precsim_A a$, then  $b\precsim_A a' \mathcal{P} x' \precsim_B x$, so $b\overline{\mathcal{P}} x$.
\end{enumerate}
Therefore, $\overline{\mathcal{P}} $ is a biorder.
We reason dually in case  $(B, \precsim_{B})$ is a chain.
\end{proof}


\begin{remark}\label{Rcom}
\noindent$(1):$
By   Proposition~\ref{Pbiorder}, it is clear that a communication $\mathcal{P}$ is a causal biorder as well as $\overline{\mathcal{P}}$ is a biorder.
\smallskip

\smallskip
\noindent$(2):$ In order to keep close to the original definition given by Lamport (see Definition~\ref{Dds}), where  sending or receiving a message is an event, in Definition~\ref{Dcomm} is not allowed to send a message from $a$ in $A$ to more than one receiver in $B$. Dually, an element $b$ in  $B$ cannot receive more than one message from $A$. From a mathematical point of view, it would be possible to generalise the idea of communication without restricting it to a finite cardinal, that is, with $|\mathcal{P}|=\infty$. However, in the present paper we shall work just on communications in the sense of Definition~\ref{Dcomm}. 

\end{remark}

 Now we are ready to  propose a definition of a  \emph{distributed system of $n$ sets} from a mathematical and theoretical point of view { (the original  definitions of these concepts --an event, a process and a distributed system-- have been introduced in Section~\ref{s1})}. 

\begin{definition}\label{definitionDS}\rm
Let $\{(X_i,\precsim_i)\}_{i=1}^n$ be a family of disjoint and  totally  ordered sets and   $\mathrel{\mathcal{P}}=\{\mathrel{\mathcal{P}_{ij}}\}_{i\neq j}$ (with $i,j\in \{1,...,n\}$)  be a family of communications from $X_i$ to $X_j$ (with $i\neq j$).  
Each totally  ordered set is said to be a \emph{process}. Each element of the processes  is said to be an \emph{event}.
The pair \mbox{$(\bigcup_{i=1}^n (X_i, \precsim_i),  \mathrel{\mathcal{P}})$} is said to be a
   \emph{distributed system}.

\end{definition}

\begin{remark}
\noindent$(1)$ In the previous definition, as it was in the original definition of L. Lamport (see \cite{Lamport}, in particular page 559 and footnote~2), the messages may be received out of order (that is, without satisfying the causal ordering of messages). Furthermore, with this definition there may be cycles with respect to the causal relation (see Figure~\ref{Fcycle}).
\smallskip


\noindent$(2)$ Notice that a  communication may be empty, so that there is no sending of messages in one direction between two processes. In this case we shall denote $\mathrel{\mathcal{P}_{ij}}=\{\emptyset\}$.

\smallskip

\noindent$(3)$ It can be proved that each total order $\precsim_i$ in $X_i$ refines the traces $\precsim_{ij}^*$ and $\precsim_{ji}^{**}$ related to the biorders $\overline{\mathcal{P}}_{ij}$ and $\overline{\mathcal{P}}_{ji}$ (respectively), for any $j\neq i$. Actually, this property was used in \cite{dsjmp} in order to define the concept of a distributed system of two processes. However, the use of the \emph{communication} concept in order to define a distributed system is closer to reality, since it captures the idea of sending and receiving messages. Moreover, it derives in the new term of \emph{causal biorder}, which seems interesting when dealing with a set endowed with more than one relation. Hence, Definition~\ref{definitionDS} has been written by means of communications.

\end{remark}

\begin{proposition}\label{Pbaliokide2}
Let $(A,\precsim_A)$ and $(B,\precsim_B)$  be  two disjoint  totally ordered sets and  $\mathcal{P}_1$ and $\mathcal{P}_2$ two communications from  $A $ to  $ B $. Let \mbox{$((A,\precsim_A)\cup (B,\precsim_B), \mathcal{P}_1)$} and   $((A,\precsim_A)\cup$ $ (B,\precsim_B), \mathcal{P}_2)$ be the corresponding distributed systems and $\rightarrow_1$ and $\rightarrow_2$ their causal relations, respectively. If the causal relations concur, i.e.,  $\rightarrow_1=\rightarrow_2$, then the communications are also the same (i.e., $\mathcal{P}_1=\mathcal{P}_2$) or  the causal ordering of messages is not satisfied.
\end{proposition}
\begin{proof}
Let $a\in A$ and $b\in B$ be such that $a\mathcal{P}_1 b$. Then, $a\rightarrow_1 b$, that means $a\rightarrow_2 b$, thus, there exist $a_1,b_1$ such that $a\precsim_A a_1\mathcal{P}_2 b_1\precsim_B b$. Thus, $a\mathcal{P}_1 b$ implies $a\overline{\mathcal{P}}_2 b$.

Suppose now that  $a\overline{\mathcal{P}}_2 b$ but $\neg( a \mathcal{P}_2 b)$. Then, there must exist $a_2\in A, \,b_2\in B$ such that  $a\prec_A a_2\mathcal{P}_2 b_2\precsim_B b$ or  $a\precsim_A a_2\mathcal{P}_2 b_2\prec_B b$. 
Assume that  $a\prec_A a_2\mathcal{P}_2 b_2\precsim_B b$ is satisfied (we reason analogously for the  dual case), then it holds that $a_2\rightarrow_2 b$ with $a\prec_A a_2$, that is, $a_2\rightarrow_1 b$ with $a\prec_A a_2$. So, $a_2\overline{\mathcal{P}}_1 b$, thus, $a\prec a_2\precsim_A a_3\mathcal{P}_1 b_3\precsim_B b$ for some $a_3\in A$, $b_3\in B$.

Here we distinguish two cases. If $b_3=b$, then $\mathcal{P}_1$ fails to be a communication since we have that $a_3\mathcal{P}_1 b$ as well as $a\mathcal{P}_1 b$, with $a\neq a_3$. If $b_3\prec b$, then the causal ordering of messages is not satisfied, since we have that $a_3\mathcal{P}_1 b_3$ and  $a \mathcal{P}_1 b $ as well as $a\prec_A a_2$ and $b_3\prec_B b$ (see Figure~\ref{Fcausalorderm}). This concludes the proof.
\end{proof}


\medskip

Now, we introduce the concept of \emph{line communication}. 

\begin{definition}\rm
Let $(\bigcup\limits_{i=1}^n (X_i, \precsim_i),  \mathrel{\mathcal{P}})$ be a distributed system, where   $\mathrel{\mathcal{P}}=\{\mathrel{\mathcal{P}_{ij}}\}_{i\neq j}$ (with $i,j\in \{1,...,n\}$)  is the family of communications from $X_i$ to $X_j$ (with $i\neq j$). 
It is said that  $\mathrel{\mathcal{P}}$ is a \emph{line communication} if $\mathcal{P}_{ij}=\{\emptyset\}$ for any $j\neq i+1$, for each $i=1,...,n-1$.
\end{definition}
Hence, when the processes are endowed with a line communication, these computers or processes are ordered in a sequence (so, totally ordered) such that each computer only sends messages to the next one (see Figure~\ref{ds3p}).

\begin{center}
\begin{figure}
\begin{center}

\begin{tikzpicture}[scale=0.55]
\begin{scope}[very thick]

\begin{scope}[very thick]

\draw[fill=gray!20!white] (-3,0) 
ellipse (48pt and 90pt);

\draw[fill=gray!40!white] (2.5,0) 
ellipse (48pt and 90pt);

\draw[fill=gray!40!white] (8,0) 
ellipse (48pt and 90pt);

\end{scope}
\end{scope}
\draw[red, -latex] (-3,0) -- (2.5,1);
\draw[red, -latex] (-3,1.2) -- (2.5,1.7);
\draw[red, -latex] (-3,-2) -- (2.5,0);
\draw[blue, -latex] (2.5,-0.6) -- (8,0.5);
\draw[blue, -latex] (2.5,1) -- (8,1.3);
\draw[blue, -latex] (2.5,-2) -- (8,-1);

\draw[green, -latex] (-3,-2.7) -- (-3,2);
\draw[green, -latex] (2.5,-2.7) -- (2.5,2);
\draw[green, -latex] (8,-2.7) -- (8,2);
\draw (-1.75,2.3) node[anchor=east] {  \small{$(X_1, \precsim_1)$}} ;
\draw (3.85,2.3) node[anchor=east] {  \small{$(X_2, \precsim_2)$}} ;
\draw (9.35,2.3) node[anchor=east] {  \small{$(X_3, \precsim_3)$}} ;
\end{tikzpicture}\end{center}
\medskip
$ $
\begin{center}
\caption{A distributed system of three processes with line communication.}\end{center}
\label{ds3p}
\end{figure}
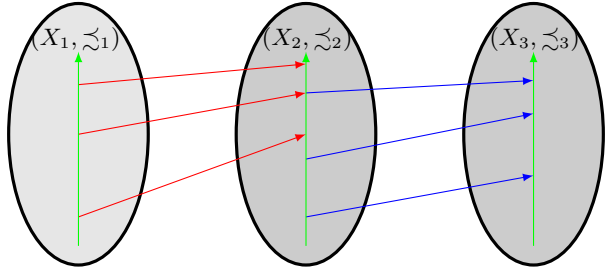
\end{center}

 
\begin{proposition}
Let 
$(\bigcup\limits_{i=1}^n (X_i, \precsim_i),  \mathrel{\mathcal{P}}=\bigcup\limits_{i\neq j} \mathrel{\mathcal{P}_{ij}}  )$ 
and 
$(\bigcup\limits_{i=1}^n (X_i, \precsim_i),  \mathrel{\mathcal{P}'}=\bigcup\limits_{i\neq j} \mathrel{\mathcal{P}'_{ij}}  )$ 
  be two distributed systems with the same processes and both with line communication. Assume  that the causal ordering of messages is  always satisfied as well as there is no cycles (that is, the $\rightarrow$ is a strict partial order).
Then, the corresponding causal precedences coincide ($\rightarrow_1=\rightarrow_2$) if and only if $\mathrel{\mathcal{P}}_{ij}=\mathrel{\mathcal{P}'}_{ij}$  for each $i\neq j$, that is, if and only if they have the same communications. 
\end{proposition}
\begin{proof}
\noindent$\Rightarrow:$ If the corresponding causal precendeces coincide, then they also coincide when we restrict the relation to a subset $X_i\cup X_{i+1}$ (for any $i=1,...,n-1$). Moreover, since $ \mathrel{\mathcal{P}}$ is a line communication, it holds that $\rightarrow_{1{|_{X_i\cup X_{i+1}}}}=(\precsim_i \cup \precsim_{i+1}\cup \mathrel{\mathcal{P}}_{i\, i+1})^+$, that is, the restriction of the causal relation $\rightarrow_1$ to $X_i\cup X_{i\, i+1}$ is just the causal relation of the distributed system $((X_i,\precsim_i)\cup (X_{i+1},\precsim_{i+1}), \mathcal{P}_{i\, i+1})$. Dually, it holds that \mbox{$\rightarrow_{2{|_{X_i\cup X_{i+1}}}}$}$=(\precsim_i \cup \precsim_{i+1}\cup \mathrel{\mathcal{P}}_{i\, i+1}')^+$.
Therefore, 
by Proposition~\ref{Pbaliokide2}, the communications $\mathrel{\mathcal{P}}_{i\,i+1}$ and $\mathrel{\mathcal{P}'}_{i\, i+1}$ 
 also coincide, and that for any $i=1,...,n-1$.


\noindent$\Leftarrow:$ This implication is trivial.
\end{proof}

\begin{remark}\label{Rnotunique}
Notice that the descomposition of a partially ordered set in $n$ disjoint chains is not unique. Therefore,   it may be possible to construct distinct distributed systems such that the corresponding causal relation coincides with the  initial partial order. 
In order to show that we include the following   example:

\end{remark}

\begin{example}\label{Enotunique}\rm
Let $\precsim$ be a partial order defined on $X=\{a,b,c,d\}$ by $\{ c\precsim b\precsim a, b\precsim d \}$.
Then, the partial order can be characterized by means of the following distributed systems (among others) of 2 processes (see Figure~\ref{Fnotunique}):

\noindent$(1)$
$(X_1,\precsim_1)= (\{ a,b,c\}, c\prec_1 b\prec_1 a) $ and $  (X_2, \precsim_2) =(\{d\}, \{\emptyset\} )$, with communication $\mathcal{P}_{12}=\{(b, d)\}$.

\noindent$(2)$
$(X_1,\precsim_1)= (\{ a,c\}, c\prec_1  a) $ and $ (X_2, \precsim_2) =(\{b,d\}, b\prec_2 d)$, with communications $\mathcal{P}_{12}=\{(c, b)\}$ and $\mathcal{P}_{21}=\{(b, a)\}$.

 \begin{figure}[htbp]
\begin{center} 
\begin{tikzpicture}[scale=1.2] 
\draw[fill=gray!20!white] (0,1) 
ellipse (20pt and 45pt);

\draw[fill=gray!20!white] (2,1) 
ellipse (20pt and 45pt);

  \draw[] (0,0)--(0,1);
    \draw[] (0,1)--(0,2);
        \draw[dashed, -latex] (0,1)--(2,2);
\draw (-0.1,2) node[anchor=east] {\small $a$};  
\draw (-0.1,1) node[anchor=east] {\small $b$};  
\draw (-0.1,0) node[anchor=east] {\small $c$};  
\draw (2.45,2) node[anchor=east] {\small $d$};  

\draw (0.17,2) node[anchor=east] {$\bullet$};  
\draw (0.17,1) node[anchor=east] {$\bullet$};   
\draw (0.17,0) node[anchor=east] {$\bullet$};  
\draw (2.2,2) node[anchor=east] {$\bullet$};  
\end{tikzpicture}\qquad\qquad
\begin{tikzpicture}[scale=1.2] 
\draw[fill=gray!20!white] (0,1) 
ellipse (20pt and 45pt);

\draw[fill=gray!20!white] (2,1) 
ellipse (20pt and 45pt);

    \draw[] (0,0)--(0,2);
    \draw[] (2,1)--(2,2);
        \draw[dashed, -latex] (2,1)--(0,2);
        \draw[dashed, -latex] (0,0)--(2,1);
\draw (0,2) node[anchor=east] {\small $a$};  
\draw (0,0) node[anchor=east] {\small $c$};  
\draw (2.45,2) node[anchor=east] {\small $d$};  
\draw (2.45,1) node[anchor=east] {\small $b$};  

\draw (0.17,2) node[anchor=east] {$\bullet$};  
\draw (0.17,0) node[anchor=east] {$\bullet$};   
\draw (2.2,2) node[anchor=east] {$\bullet$};  
\draw (2.2,1) node[anchor=east] {$\bullet$};  
\end{tikzpicture}\\
\medskip
$ $
\caption{A partial order represented through two distinct distributed systems.}
\label{Fnotunique}
\end{center}
\end{figure}
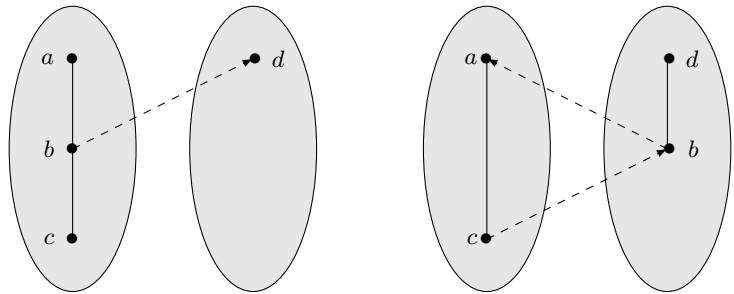

\end{example}

\begin{remark}\label{Rnetwork}

Hence, given a preorder, it seems interesting to study the existence and uniqueness of distributed systems that characterize (through its causal precedence)  the preorder with some additional properties such as that the length of the processes is minimal, or that the number of messages is minimal, etc.  For instance, in the example of Figure~\ref{Fnotunique}, in the first case the length of the longest process is three and the number of messages is one, whereas in the second case these values are two and two, respectively. 

\end{remark}

\section{Representability of Distributed Systems}\label{srepds}

 Since it is possible to add a new process  to a distributed system (that is, connecting another computer to the system, including also the corresponding communication), it is interesting to study how to create a new representation of the distributed system arised from the union of two distributed systems, but now aggregating the representations before. 

In this paper we do not achieve the answer to this question but, at least, we are able to construct weak representations of distributed systems with line communications starting from pairs of functions that represent each biorder. 

We shall assume that the causal ordering of messages is satisfied, as well as there are no cycles. Thus, we assume that our distributed systems are such that the causal relation $ \rightarrow $ is a strict partial order.

\medskip

The following  definitions introduce the concept of representability for a distributed system of $n$ processes.

\begin{definition}\label{weakly}\rm
Let 
$(\bigcup_{i=1}^n (X_i, \precsim_i),  \mathrel{\mathcal{P}}=\bigcup_{i\neq j} \mathrel{\mathcal{P}_{ij}}  )$ 
  be  a distributed system. We say that it is \emph{weakly representable} if there exists a family $\{u_i\}_{i=1}^n$ (called \emph{weak representation}) of real functions $u_i\colon X_i\to \mathbb{R}$
such that $(u_i,u_j)$ weakly represents the biorder $\mathrel{\overline{\mathcal{P}}_{ij}}$ with respect to $<$ (that is, $x_i \mathrel{\overline{\mathcal{P}}_{ij}} x_j \Rightarrow  u_{i}(x_i)<u_{j}(x_j)$, for any $x_i\in X_i, \, x_j\in X_j$)  as well as each $u_i$  represents the total order $\precsim_i$ (i.e., $x\precsim_i y \iff u_i(x)\leq u_i(y), \, x,y\in X_i$), for any $i,j\in \{1,...,n\} $ and $i\neq j$.

If each set $X_i$ is endowed with a topology $\tau_i$, then we will say that the distributed system is \emph{continuously weakly  representable} if there exists a continuous weak representation.
\end{definition}
\begin{remark}
\noindent $(1)$
Notice that, for any $x\in X_i$ and $y\in X_j$ such that $x\rightarrow y$, it holds that $u_i(x)<u_j(y)$, for any $i,j\in \{1,2,...,n\}$.

\smallskip
\noindent $(2)$
The aforementioned functions $u_i$ are known as \emph{Lamport clocks} (see \cite{Lamport}). In fact, a Lamport  clock is a function $\mathcal{C}$ satisfying that $a\rightarrow b$ implies $\mathcal{C}(a)<\mathcal{C}(b)$, for any $a,b\in X$. On the other hand,  $\mathcal{C}(a)<\mathcal{C}(b)$ does not imply  $a\rightarrow b$.
\end{remark}

\begin{definition}\label{representable}\rm
Let 
$(\bigcup_{i=1}^n (X_i, \precsim_i),  \mathrel{\mathcal{P}}=\bigcup_{i\neq j} \mathrel{\mathcal{P}_{ij}}  )$ 
  be  a distributed system. We say that it is \emph{(finitely) representable} if there exists a (finite) family of weak representations $\{ \{u_{i}^k\}_{i=1}^n\}_{k\in \mathcal{K}}$ 
such that 
 $x_i \mathrel{\overline{\mathcal{P}}_{ij}} x_j $ iff $ u^k_{i}(x_i)<u_{j}^k(x_j)$ for any $k\in \mathcal{K}$, for any $x_i\in X_i, \, x_j\in X_j$ and $i\neq j$.

If each set $X_i$ is endowed with a topology $\tau_i$, then we will say that the distributed system is \emph{continuously  representable} if there exists a continuous  representation.
\end{definition}

As following Preposition~\ref{Pduality} shows, the term before is analogous 
 to the Richter-Peleg multi-utility representation used for preorders. \cite{partial,BH,EvO,Peleg,Richter}

\begin{proposition}\label{Pduality}
A distributed system 
$(\bigcup_{i=1}^n (X_i, \precsim_i),  \mathrel{\mathcal{P}}=\bigcup_{i\neq j} \mathrel{\mathcal{P}_{ij}}  )$  is (finitely) representable if and only if the corresponding  causal relation  $\rightarrow$   is (finitely) Richter-Peleg multi-utility representable.
\end{proposition}
\begin{proof}
Given a representation $\{ \{u_{i}^k\}_{i=1}^n\}_{k\in \mathcal{K}}$ of the distributed system, the family of functions $\{w_k\}_{k\in \mathcal{K}}$ defined by 
\[ w_k(x)=u_{i}^k(x) \text{ if } x\in X_i, \text{ with } k\in \mathcal{K},\]
is a Richter-Peleg multi-utility representation of the strict partial order $\rightarrow$.

Dually, starting from a Richter-Peleg  multi-utility representation $\{ w_k(x)\}_{k\in \mathcal{K}}$ of the causal relation  $\rightarrow$ of a distributed system 
$(\bigcup_{i=1}^n (X_i, \precsim_i),  \mathrel{\mathcal{P}}=\bigcup_{i\neq j} \mathrel{\mathcal{P}_{ij}}  )$, then the following representation  $\{ \{u_{i}^k\}_{i=1}^n\}_{k\in \mathcal{K}}$ arises:
 \[ u_{i}^k(x)=w_k(x),  \text{ when } x\in X_i, \text{ with } k\in \mathcal{K}, \text{ for each } i=1,...,n.\]\end{proof}

\begin{remark}
Notice that the main difference is the domain of the corresponding functions. In the case of distributed systems the functions are defined on the processes, whereas in the case of  preorders they are defined on all the set (which would be the union of the processes). This difference may be relevant when dealing with continuity and topological spaces. 
\end{remark}

Now we introduce another kind of representation in bijection 
 with the concept of multi-utility  \cite{partial,BH,EvO,Peleg,Richter},  and that it is actually common and known in computing by   \emph{vector clock representation}  (see \cite{virtual,ray}). Due to that coincidence (and in order to distinguish it from the definition before), we shall call it by \emph{vector representation}.

\begin{definition}\label{vrepresentable}\rm
Let $(\bigcup_{i=1}^n (X_i, \precsim_i),  \mathrel{\mathcal{P}}=\bigcup_{i\neq j} \mathrel{\mathcal{P}_{ij}}  )$ 
  be  a distributed system. We say that it is \emph{vector representable} if there exists a family of weaks representations   with respect to $\leq$  (called \emph{vector clocks}) $\{ \{u_{i}^k\}_{i=1}^n\}_{k\in \mathcal{K}}$  such that 
 $x_i \mathrel{\overline{\mathcal{P}}_{ij}} x_j $ iff $ u_{i}^k(x_i)\leq u_{j}^k(x_j)$ for any $k\in \mathcal{K}$, as well as there exists an index $l\in \mathcal{K}$ such that $u_{i}^l(x_i)< u_{j}^l(x_j)$ (for any $x_i\in X_i, \, x_j\in X_j$ and $i\neq j$).

If each set $X_i$ is endowed with a topology $\tau_i$, then we will say that the distributed system is \emph{continuously vector representable} if there exists a continuous  vector representation.
\end{definition}
\begin{remark}
In computing, these vector clocks are constructed through \emph{timestamps} algorithms (see \cite{virtual,ray}).
\end{remark}

The problem of aggregating representations is not trivial. In order to ilustrate that, we include the following example.

\begin{example}\label{exkontra}\rm
Let $\mathcal{P}_{12}$ and $\mathcal{P}_{23}$ two communications between $A=\{a,b\}$ (with $a\prec_A b$) and $X=\{x,y\}$  (with $x\prec_X y$) and $\Lambda=\{\alpha, \beta \}$  (with $\alpha\prec_{\Lambda} \beta$), respectively, defined as follows:
$$a\mathcal{P}_{12} x\mathcal{P}_{23} \beta, \quad a\mathcal{P}_{12} y \quad\text {and}\quad b\mathcal{P}_{12} y.$$
Now we define the tuple   $(u,v,w)$ by $u(a)=0, u(b)=1, v(x)=1, v(y)=2, w(\alpha)=1$ and $w(\beta)=2$.
Then, the  pairs $(u,v)$ and $(v,w)$ are representations  of the distributed systems defined on $A\cup X$ and on $X\cup \Lambda$, respectively.
However, the tuple  $(u,v,w)$  fails to represent the distributed system made up by the three processes:
  $u(a)=0<w(\alpha)=1$ but $\neg (a\rightarrow \alpha)$.
\end{example}

 \begin{figure}[htbp]
\begin{center} 
\begin{tikzpicture}[scale=1.2] 
\draw[fill=gray!20!white] (0,1.5) 
ellipse (10pt and 30pt);

\draw[fill=gray!20!white] (1,2) 
ellipse (10pt and 30pt);
\draw[fill=gray!20!white] (2,1.5) 
ellipse (10pt and 30pt);
    \draw[] (1,1.5)--(1,2.5);
    \draw[] (0,1)--(0,2);
    \draw[] (2,1)--(2,2);
    \draw[dashed, -latex] (0,1)--(1,1.5);
    \draw[dashed, -latex] (0,2)--(1,2.5);
    \draw[dashed, -latex] (1,1.5)--(2,2);

\draw (0,2) node[anchor=east] {\small $b$};  
\draw (0,1) node[anchor=east] {\small $a$};   
\draw (2.35,1.9) node[anchor=east] {\small $\beta$};  
\draw (2.35,1.1) node[anchor=east] {\small $\alpha$};  
\draw (1.2,2.7) node[anchor=east] {\small $y$};
\draw (1.2,1.3) node[anchor=east] {\small $x$};

\draw (0.2,2) node[anchor=east] {$\bullet$};  
\draw (0.2,1) node[anchor=east] {$\bullet$};   
\draw (2.2,2) node[anchor=east] {$\bullet$};  
\draw (2.2,1) node[anchor=east] {$\bullet$};  
\draw (1.2,2.5) node[anchor=east] {$\bullet$};
\draw (1.2,1.5) node[anchor=east] {$\bullet$};
\end{tikzpicture}\qquad\qquad
\begin{tikzpicture}[scale=1.2] 
\draw[fill=gray!20!white] (0,1.5) 
ellipse (10pt and 30pt);

\draw[fill=gray!20!white] (1,2) 
ellipse (10pt and 30pt);
\draw[fill=gray!20!white] (2,1.5) 
ellipse (10pt and 30pt);
    \draw[] (1,1.5)--(1,2.5);
    \draw[] (0,1)--(0,2);
    \draw[] (2,1)--(2,2);
    \draw[dashed, -latex] (0,1)--(1,1.5);
    \draw[dashed, -latex] (0,2)--(1,2.5);
    \draw[dashed, -latex] (1,1.5)--(2,2);

\draw (0,2) node[anchor=east] {\small $1$};  
\draw (0,1) node[anchor=east] {\small $0$};   
\draw (2.35,1.9) node[anchor=east] {\small $2$};  
\draw (2.35,1.1) node[anchor=east] {\small $1$};  
\draw (1.2,2.7) node[anchor=east] {\small $2$};
\draw (1.2,1.3) node[anchor=east] {\small $1$};

\draw (0.2,2) node[anchor=east] {$\bullet$};  
\draw (0.2,1) node[anchor=east] {$\bullet$};   
\draw (2.2,2) node[anchor=east] {$\bullet$};  
\draw (2.2,1) node[anchor=east] {$\bullet$};  
\draw (1.2,2.5) node[anchor=east] {$\bullet$};
\draw (1.2,1.5) node[anchor=east] {$\bullet$};

\end{tikzpicture}\\
\medskip
$ $
\caption{A distributed system. Each process contains just two events. The dashed arrows represent the communication between processes.}
\label{Fcounter}
\end{center}
\end{figure}
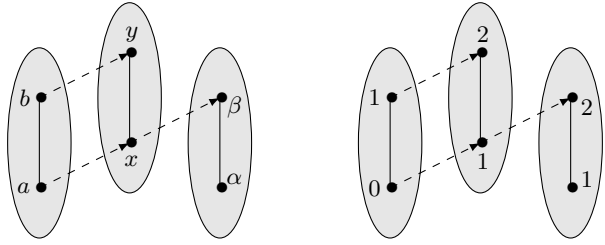

In the following lines it is shown how to construct   weak representations of distributed systems with line communications starting from pairs of functions that represent each biorder.  For more clearness, before introduce the general case, first we include here the particular case of a distributed system of three processes with a linear communication.

\begin{proposition}\label{proptupleweak2}
Let $(\bigcup_{i=1}^3 (X_i, \precsim_i),  \mathrel{\mathcal{P}}=\bigcup_{i=1 }^{2} \mathrel{\mathcal{P}_{i\, i+1}}  )$ a distributed system of three processes with a linear communication such that for every  $i\in \{1,2\}$ the pair $(u_i,v_i)$ is a representation of the biorder $\mathrel{\overline{\mathcal{P}}_{i\,i+1}}$\footnote{We may assume, without loss of generality,  that the codomain of the utilities is the interval $(0,1)$ instead of the all real line.}. Then, 
the tuple $(u=u_1+u_{u_2}, v=v_1+u_2,w=v_2+v_{v_1})$ is a weak representation  of the distributed system, where $u_{u_2}$ and $v_{v_1}$ are defined  as follows:
$$u_{u_2}(x)=\inf\{u_2(y)\colon x\overline{\mathcal{P}}_{12} y\, ; \, y\in X_2\}; \text{ for any } x\in X_1,$$
$$v_{v_1}(x)=\sup\{v_1(y)\colon y\overline{\mathcal{P}}_{23} x\, ; \, y\in X_2\}; \text{ for any } x\in X_3,$$
on the assumption that $\inf\{\emptyset\}=1$ and $\sup\{\emptyset\}=0.$

\end{proposition}
\begin{proof}
First, notice that since the utilities take values on $(0,1)$, and taking into account that $\inf\{\emptyset\}=1$ and $\sup\{\emptyset\}=0$,  the functions $u_{u_2}$ and $v_{v_1}$ are well defined (that is, the infimum and the supremum always exist).

Let $x,y$ be two elements such that $x\overline{\mathcal{P}}_{12} y$. Then, it holds true that $u_1(x)<v_1(y)$ as well as $u_{u_2}(x)\leq u_2(y)$. Therefore, the condition $u(x)<v(y)$ is satisfied.
We argue analogously for a pair of elements $x,y$ such that $x\overline{\mathcal{P}}_{23} y$.


Since $x\precsim_1 x' \overline{\mathcal{P}}_{12} y$ implies $x\overline{\mathcal{P}}_{12} y$, the inequality 
$u_{u_2}(x)\leq u_{u_2}(x')$ is also satisfied. 
If $x\precsim_1 x'$ and there is no $y\in X_2$ such that $x' \overline{\mathcal{P}}_{12} y$, then $u_{u_2}(x')=\inf\{\emptyset\}=1$ so, again, it holds that $u_{u_2}(x)\leq u_{u_2}(x')$. Thus, we conclude that $u_{u_2}(x)\leq u_{u_2}(x')$ is always satisfied 
 for any $x,x'\in X_1$ such that $x\precsim_1 x'$. We argue analogously for a pair of elements $y,y'\in X_3$ such that $y\precsim_{3} y'$.
In addition, the functions $v_1$ and $u_2 $ also represent the total order $\precsim_2$, as well as $u_1$ and $v_2$ represent the total orders $\precsim_1$ and $\precsim_3$, respectively. Hence, we deduce that the functions $u=u_1+u_{u_2}, v=v_1+u_2$ and $w=v_2+v_{v_1}$ are representations of the total orders $\precsim_1$, $\precsim_2$ and $\precsim_3$, respectively.

Let $x,z$ be now two elements such that $x\overline{\mathcal{P}}_{12} y \overline{\mathcal{P}}_{23}  z$, for some $y\in X_2$. Then,  since it holds true that  $u_1(x)<v_1(y)\leq v_{v_1}(z)$ and $u_{u_2}(x)\leq u_2(y)\leq v_{2}(z)$, the condition $u(x)<w(z)$ is satisfied.

Thus, we conclude that $(u=u_1+u_{u_2}, v=v_1+u_2,w=v_2+v_{v_1})$ is a weak representation of the distributed system.
\end{proof}

Before generalize the proposition above to $n$ processes, first we introduce the following operators.

\begin{definition}\label{defoperator}\rm
Let   \mbox{$(\bigcup_{i=1}^2 (X_i, \precsim_i),  \mathrel{\mathcal{P}} )$} be  a distributed system with a single communication $\mathcal{P}$ from $X_1$ to $X_2$.
Let $u $ and $v$ be two (not necessarily strictly) increasing functions on $X_1$ and $X_2$ (respectively) that take values on $(0,1)$. 
Assume  that $\inf\{\emptyset\}=1$ and $\sup\{\emptyset\}=0.$
Then, we define the following two operators:

$$ \underline{op}(v)(x)=\inf \{v(y) \colon x\overline{\mathcal{P}} y\}_{\{y\in X_2\}}\, ; \quad x\in X_1, $$
$$ \overline{op}(u)(x)=\sup \{u(y) \colon y\overline{\mathcal{P}} x\}_{\{y\in X_1\}}\, ; \quad x\in X_2. $$

We shall call these operators  \emph{lower operator} and \emph{upper operator}, respectively.

\end{definition}

\begin{remark}\label{Rclave}
\noindent$(1)$
Notice that, starting from a function $u$ on $X_1$, $\overline{op}(u)$ defines a new function on $X_2$, and not in $X_1$. Dually, starting now from a function $v$ on $X_2$, $\underline{op}(v)$ defines a new function on $X_1$, and not in $X_2$. 
\smallskip

\noindent$(2)$ In fact, since sending and receiving messages is an event, the infimum (of a non-empty set) is a minimum and the suprema (of a non-empty set) is a maximum.  Otherwise, the infimum of an empty set is the top of the ordered set, that is, 1, and the supremum is the bottom, that is, 0. 
\end{remark}

\begin{proposition}\label{Poppro}
Let   $(\bigcup_{i=1}^2 (X_i, \precsim_i),  \mathrel{\mathcal{P}} )$ be a distributed system with a single communication $\mathcal{P}$ from $X_1$ to $X_2$.
Let $u $ and $v$ be two (not necessarily strictly) increasing functions on $X_1$ and $X_2$ (respectively) that take values on $(0,1)$.
Then, the following properties are satisfied:
\begin{enumerate}
\item[($i$)] $\underline{op}(v)$ and $\overline{op}(u)$ are increasing in $X_1$ and $X_2$, respectively.
\item[($ii$)] The pairs $(\underline{op}(v), v )$ and $(u, \overline{op}(u))$ represent the biorder $\overline{\mathcal{P}}$ with respect to $\leq$.
\item[$(iii)$]  $x_1\sim^*y_1$ implies $\underline{op}(v)(x_1)=\underline{op}(v)(y_1)$, as well as $x_2\sim^{**}y_2$ implies $\overline{op}(v)(x_2)=\overline{op}(v)(y_2)$, for any $x_1,y_1\in X_1, \, x_2,y_2\in X_2$.\footnote{Here, $\sim^*$ and $\sim^{**}$ denote the equivalence relations associated to the traces $ \precsim^*$ and $ \precsim^{**}$ of the biorder $\mathrel{\overline{\mathcal{P}}} $, respectively.}
\end{enumerate}
\end{proposition}
\begin{proof}
\noindent$(i):$ 
If $x\precsim_1 y$, then   it holds that $y\overline{\mathcal{P}} z $ implies $x\overline{\mathcal{P}} z $. Therefore,  applying the definition of the lower operator, it follows that the inequality $\underline{op}(v)(x)\leq \underline{op}(v)(y)$ is satisfied, that is, $\underline{op}(v)$ is increasing with respect to $\precsim_1$ on $X_1$.
We argue dually in order to prove that $\overline{op}(u)$ is increasing with respect to the total order $\precsim_{2}$ defined on $X_{2}$.
\medskip

\noindent$(ii):$ 
If $x\overline{\mathcal{P}} y$, then  applying the definition of the lower operator, it is clear that the inequality $\underline{op}(v)(x)\leq (v)(y)$ is satisfied. We argue analogously for the pair $(u, \overline{op}(u))$.
On the other hand, suppose that $\underline{op}(v)(x)\leq v(y)$. Since $v$ takes values on $(0,1)$, $\underline{op}(v)(x)=r\in (0,1)$, which means that (by definition of $\underline{op}(v)$) there exists $z\in X_2$ such that $x\overline{\mathcal{P}} z$ with $v(z)=r$ (here take into account Remark~\ref{Rclave} (2)). Thus, $v(z)=r\leq v(y)$ and then, $z\precsim_2 y$. Therefore, we conclude that $x\overline{\mathcal{P}} y$. 
Hence, the pair $(\underline{op}(v), v )$  represents the biorder with respect to $\leq$.
\medskip

\noindent$(iii):$ 
If  $x_1\sim^*y_1$, then $x_1\overline{\mathcal{P}} z$ holds if and only if $y_1\overline{\mathcal{P}} z$  is satisfied, for any $z\in X_2$. Therefore, by the definition of the lower opertator, the equality $\underline{op}(v)(x_1)= \underline{op}(v)(y_1)$ holds true.  
We argue analogously for  the indifference  $\sim^{**}$.
\end{proof}

\begin{remark}\label{Rop}

\noindent$(1)$
Dealing with a distributed system  $(\bigcup_{i=1}^n (X_i, \precsim_i),  \mathrel{\mathcal{P}}=\bigcup_{i\neq j } \mathrel{\mathcal{P}_{ij}}  )$ of $n$ processes, since --by Proposition~\ref{Poppro} $(i)$-- $\underline{op}(u_i)$ and $\overline{op}(u_i)$ are increasing in their corresponding sets ($X_{i-1}$ and $X_{i+1}$, respectively), it 
is possible to apply an operator more than once. So, starting from an increasing function $u_i$ in $X_i$,
we shall denote by $\underline{op}^2(u_i)$  the function $\underline{op}(\underline{op}(u_i))$ defined in $X_{i-2}$.
This notation is generalized to  $\underline{op}^k(u_i)$, achieving a function in $X_{i-k}$.
We shall use the same notation for the upper operator $\overline{op}$.
Since the hypothesis of Proposition~\ref{Poppro} are again satisfied (now for $\underline{op}^k(u_i)$ and $\overline{op}^k(u_i)$), the proporties $(i)$ and ($ii$) are also true for these iterations.
\medskip

\noindent$(2)$
The fact that the functions $u$ and  $v$ are strictly increasing (thus, they represent the corresponding total preorder)   does not guarantee that $\underline{op}(v)$ and $\overline{op}(u)$ are also.
\end{remark}

\begin{theorem}\label{proptupleweakn}
Let $(\bigcup_{i=1}^n (X_i, \precsim_i),  \mathrel{\mathcal{P}}=\bigcup_{i=1 }^{n-1} \mathrel{\mathcal{P}_{i\, i+1}}  )$ a distributed system of $n$ processes with a line  communication 
  such that for every  $i\in \{1,...,n-1 \}$ the pair $(u_i,v_i)$ is a representation of the biorder $\overline{\mathcal{P}}_{i\,i+1}$, with the additional property that $u_i$ and $v_i$ represent the total  orders $\precsim_i$ and $\precsim_{i+1}$, respectively.
Then, 
 $(w_1,...,w_{n})$ is a  weak representation of the distributed system, where each function $w_i$ is defined on $X_i$ by a sum of $n-1$ functions as follows:
 \begin{center}
 \begin{tabular}{ccc}
 $w_1$ & $=$ & $u_1+ \sum_{k=1}^{n-2} \underline{op}^k(u_{k+1})$\\
  $w_2$ & $=$ & $v_1+ u_2+ \sum_{k=1}^{n-3} \underline{op}^k(u_{k+2})$\\
   $w_3$ & $=$ & $\overline{op}(v_1)+ v_2+u_3+ \sum_{k=1}^{n-4} \underline{op}^k(u_{k+3})$\\
   $\vdots$ & $\vdots$ & $\vdots$\\
    $w_j$ & $=$ &  $\sum_{k=1}^{j-2} \overline{op}^{j-1-k}(v_{k}) + v_{j-1} +u_j+  \sum_{k=1}^{n-j-1} \underline{op}^k(u_{k+j})$\\
   $\vdots$ & $\vdots$ & $\vdots$\\
    $w_n$ & $=$ & $v_{n-1}+ \sum_{k=1}^{n-2} \overline{op}^k(v_{n-1-k})$
 \end{tabular}
 \end{center}
 
\end{theorem}

\begin{proof}

First, in the following table we recover the distinct functions defined on each process:

\begin{center}
\begin{tabular}{|c|c|c|c|c|c|}\hline
$X_1$ & $X_2$ & $ X_3$ & $\cdots $ & $X_{n-1}$ & $X_n$\\ \hline
 $u_1$ & $v_1$ & $\overline{op}(v_1)$ & $\cdots$ &  $\overline{op}^{n-3}(v_1)$ & $ \overline{op}^{n-2}(v_1)$\\ \hline
$\underline{op}(u_2)$ & $u_2$ & $v_2$ & $\cdots$ & $\overline{op}^{n-4}(v_2)$ & $ \overline{op}^{n-3}(v_2)$\\ \hline
$\underline{op}^2(u_3)$& $\underline{op}(u_3)$  & $u_3$ &  $\cdots$ & $\overline{op}^{n-5}(v_3)$ & $ \overline{op}^{n-4}(v_3)$\\ \hline
 $\cdots$ & $\cdots$ & $\cdots$ & $\cdots$ & $\cdots$ & $\cdots$\\ \hline
 $\underline{op}^{n-2}(u_{n-1})$& $\underline{op}^{n-3}(u_{n-1})$  & $\underline{op}^{n-4}(u_{n-1})$ &  $\cdots$ & $u_{n-1}$ & $v_{n-1}$\\ \hline
\end{tabular}

\end{center}

So, then, each function $w_i$ is the sum of all  these $n-1$ functions defined on the set $X_i$.
Let's see now that this tuple $(w_1,...,w_{n})$ is a  weak representation of the distributed system.

Since --by Proposition~\ref{Poppro} $(i)$-- all the functions defined on each set $X_i$ (for each $i\in \{1,...,n\}$) are increasing (with respect to the corresponding total  order $\precsim_i$) and there is --at least-- one which is strictly increasing ($u_i$ and/or $v_{i-1}$), we conclude that the sum of all of them (denoted by $w_i$) is a representation of the total  order $\precsim_i$.

Finally, taking into account Proposition~\ref{Poppro} $(ii)$  and that $(u_i,v_i)$ is a representation of the biorder $\overline{\mathcal{P}}_{i\,i+1} $ (for each $i\in \{1, ..., n-1\}$) with respect to $<$, it is straightward to check that if $x\overline{\mathcal{P}}_{i\, i+1} y$ holds then  $w_i(x)< w_{i+1}(y)$ is satisfied (for each $i\in \{1, ..., n-1\}$ and for any $x\in X_i, \, y\in X_{i+1}$). So, we conclude that $(w_1,...,w_{n})$ is a  weak representation of the distributed system.
\end{proof}

\section{Quasi-finite partial orders}\label{snf}

In this section a particular but interesting class of partial orders are studied: \emph{quasi-finite partial orders}.
This kind of structures includes all those partial orders that can be understood as a finite family of chains with a 
 communication. 
The key of this section is to focus the research on the quotient sets (with respect to the traces of the biorders), which makes possible a discrete study of the representability, achieving results not only of the quotient structure but also of the original one. Thus, it is also possible to apply some techniques on finite posets as those introduced in \cite{qm}.

In fact, given a distributed system 
$(\bigcup_{i=1}^n (X_i, \precsim_i),  \mathrel{\mathcal{P}}=\bigcup_{i\neq j} \mathrel{\mathcal{P}_{ij}}  )$, we may define   an equivalence relation $\mathcal{I}_i$ on $X_i$ by means of  the intersection of all the equivalence relations $\mathcal{I}_{ij}^*$ and $\mathcal{I}_{ji}^{**} $ (for any $ i\neq j$) on $X_i$ (that is, the equivalence relation associated to the union of  all the traces on $X_i$) (see Remark~\ref{rty}). 
  Then, since the communication is a finite relation, 
   the cardinal of each quotient set $\overline{X}_i=X_i \slash \mathcal{I}_i $ is finite, for any $i=1, ..., n$.

The goal of the present section is the attainment of a method to construct  finite Richter-Peleg multi-utility representations for quasi-finite partial orders, i.e., a representation method for distributed systems. 
For more clarity, 
 Example~\ref{Equotient2} is included in order to show this procedure.
\medskip

Let's see how  quasi-finite partial orders are defined.

\begin{definition}\label{Dnearfinite}\rm
  We shall say that a partial order on a set is  \emph{quasi-finite} if it is the causal relation of a  distributed system.
\end{definition}

\begin{remark}
\noindent $(1)$ By definition, quasi-finite partial orders are near-complete.
\medskip

\noindent $(2)$ Given a    distributed system 
$(\bigcup_{i=1}^n (X_i, \precsim_i),  \mathrel{\mathcal{P}}=\bigcup_{i\neq j} \mathrel{\mathcal{P}_{ij}}  )$, we may be interested just in the communication  $\mathrel{\mathcal{P}}$ and skip the remaining information, that is, the total  orders $\precsim_i$. In that case, a finite poset  $(\bigcup_{i=1}^n (X_i\slash \mathcal{I}_i, \overline{\precsim}_i),  \mathrel{\mathcal{P}}=\bigcup_{i\neq j} \mathrel{\mathcal{P}_{ij}}  )$
   is achieved, where $\overline{\precsim}_i$ is the total order on $\overline{X}_i=X_i\slash \mathcal{I}_i$  and now the communication $\mathrel{\mathcal{P}}$ is restricted to the quotient sets (see Example~11 in \cite{dsjmp}).
\end{remark}

The following theorem shows how to construct a Richter-Peleg multi-utility representation of a quasi-finite partial order,  just starting from a  bijective Richter-Peleg multi-utility representation (see Definition~\ref{lakeste})  of a finite poset and utilities of total preorders.

\begin{theorem}\label{Prprep}
Let $\precsim$ be a quasi-finite partial order on $X$  that coincides with the causal relation associated to a distributed system
$(\bigcup_{i=1}^n (X_i, \precsim_i),  \mathrel{\mathcal{P}}=\bigcup_{i\neq j} \mathrel{\mathcal{P}_{ij}}  )$.
 Let $\{w_i\}_{i=1}^n$ be a family of utilities\footnote{We may assume, without loss of generality,  that the codomain of the utilities is the interval $(0,1)$ instead of the all real line.} $w_i\colon (X_i, \precsim_i)\to (0,1)$ and  $\mathcal{U}=\{u_l\}_{l=1}^k$  a bijective Richter-Peleg multi-utility representation  associated to  $(\bigcup_{i=1}^n (X_i\slash \mathcal{I}_i, \overline{\precsim}_i),$ $  \mathrel{\mathcal{P}} )$.
Then, the family of functions $\mathcal{V}=\{v_l\}_{l=1}^k$ defined by
\[v_l(x)= u_l(\bar{x})+ w_i(x), \quad \text{if } x\in X_i, \forall x\in X\] is a Richter-Peleg multi-utility representation of the quasi-finite partial order $\precsim$.

\end{theorem}
\begin{proof}
\noindent$\Rightarrow: $ 
Let $x,y$ be two elements in $X$ such that $x\prec y$. 

If $x$ and $y$ belong to the same process $X_i$, then $x\prec_i y$ so, $w_i(x)<w_i(y)$. 
Since $x\prec_i y$, 
 it is also true that $x\precsim_{ij}^* y$ as well as $x\precsim_{ji}^{**} y$ for any trace defined on $X_i$. Therefore,  $u_l(\bar{x})\leq u_l(\bar{y})$ is satisfied, for any function $u_l\in \mathcal{U}$. 
Hence, we conclude that $v_l(x)<v_l(y)$ for any
$v_l\in \mathcal{V}$.

If $x$ and $y$ belong to distinct processes,  $X_i$ and $ X_j$ respectively,  then  it holds that $x\overline{\mathcal{P}}_{ij} y$ so, $u_l(\bar{x})<u_l(\bar{y})$ is satisfied for any 
 $u_l\in \mathcal{U}$.
 Therefore, since   $u_l(\bar{x})+1\leq u_l(\bar{y})$ and the codomain of the utilities is the interval $(0,1)$, we conclude that that $v_l(x)<v_l(y)$ for any $v_j\in \mathcal{V}$.
 \medskip
 
 \noindent$\Leftarrow:$ 
 Suppose that $v_l(x)<v_l(y)$ is satisfied for any $v_l\in \mathcal{V}$.
 Since   $v_l(x)= u_l(\bar{x})+ w_i(x)$ (with $x\in X_i$, for any $i=1,..., n$) and the codomain of the utilities is the interval $(0,1)$, 
  $v_l(x)<v_l(y)$ holds true if and only if $u_l(\bar{x})\leq u_l(\bar{y}) $ for any  $u_l\in \mathcal{U}$.
  
  Now, we distinguish two cases. \footnote{Here notice that, since $\mathcal{U}=\{u_l\}_{l=1}^k$ is bijective Richter-Peleg multi-utility associated to the poset arised from the quotient of the traces on the processes,  $u_{l_0}(x)=u_{l_0}(y)$ holds true (for some index $l_0$) if and only if both elements belong to the same class (and so, also to the same process). As a matter of  fact, since they belong to the same class then it holds that $u_{l}(x)=u_{l}(y)$ not only for that index $l_0$ but also for any index $l$.}
  \begin{enumerate}
  \item $u_l(\bar{x})< u_l(\bar{y}) $ for any  $u_l\in \mathcal{U}$. In that case we distinguish again two cases.
   \begin{enumerate}
  \item $x$ and $y$ belong to the same process $X_i$. In that case, since  $u_l(\bar{x})< u_l(\bar{y}) $ for any  $u_l\in \mathcal{U}$, there exists a trace $\precsim_{ij}^*$ or $\precsim_{ji}^{**}$ on $X_i$ such that $x\prec_{ij}^* y$ or $x\prec_{ji}^{**} y$. Thus, we conclude that $x\prec_i y$ and, hence, $x\prec y$.
   \item $x$ and $y$ belong to distinct processes,  $X_i$ and $ X_j$ respectively. In that case, since  $u_l(\bar{x})< u_l(\bar{y}) $ for any  $u_l\in \mathcal{U}$, there exist elements $x_{k_1}\in X_{k_1},...,x_{k_s}\in X_{k_s}$ (for some $k_1, ...,k_s\in \{1,...,n\}$)    such that $x\mathrel{\overline{\mathcal{P}}}_{i\,k_1} x_{k_1} \mathrel{\overline{\mathcal{P}}}_{k_1\,k_2}\cdots \mathrel{\overline{\mathcal{P}}}_{k_{s-1}\, k_s} x_k \mathrel{\overline{\mathcal{P}}}_{k_s\,j} y$. Thus, we conclude that $x\prec y$.
  \end{enumerate}
    \item $u_l(\bar{x})= u_l(\bar{y}) $ for any  $u_l\in \mathcal{U}$. In that case $x$ and $y$ belong to the same quotient class and, therefore, to the same process. Thus, since  $u_l(\bar{x})= u_l(\bar{y}) $ for any  $u_l\in \mathcal{U}$ and  $v_l(x)<v_l(y)$  for any $v_l\in V$, it holds that $w_i(x)<w_i(y)$. Then, we conclude that $x\prec_i y$ and, hence, $x\prec y$.
  \end{enumerate}
\end{proof}

 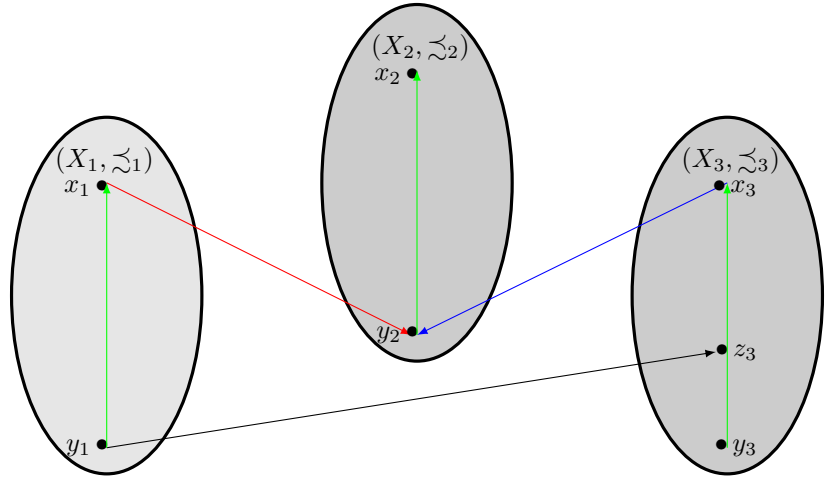
\begin{figure}[htbp]
\begin{center}
\begin{tikzpicture}[scale=0.75]
\begin{scope}[very thick]

\begin{scope}[very thick]

\draw[fill=gray!20!white] (-3,0) 
ellipse (48pt and 90pt);

\draw[fill=gray!40!white] (2.5,2) 
ellipse (48pt and 90pt);

\draw[fill=gray!40!white] (8,0) 
ellipse (48pt and 90pt);

\end{scope}
\end{scope}
\draw[red, -latex] (-3,2) -- (2.4,-0.7);
\draw[black, -latex] (-3,-2.7) -- (7.8,-1);

\draw[blue, -latex] (8,2) -- (2.5,-0.7);

    \draw (-2.8,-2.7) node[anchor=east] { $y_1\, \bullet$}; 
      \draw (-2.8,1.9) node[anchor=east] { $x_1\, \bullet$}; 
    \draw (2.7,-0.7) node[anchor=east] { $y_2\, \bullet$};   
     \draw (2.7,3.9) node[anchor=east] { $x_2\,\bullet$};    
       \draw (8.7,-2.7) node[anchor=east] { $\bullet \, y_3$};  
  \draw (8.7,-1) node[anchor=east] { $\bullet \, z_3$};
     \draw (8.7,1.9) node[anchor=east] { $ \bullet \, x_3$}; 

\draw[green, -latex] (-3,-2.7) -- (-3,2);
\draw[green, -latex] (2.5,-0.7) -- (2.5,4);
\draw[green, -latex] (8,-2.7) -- (8,2);
\draw (-2,2.4) node[anchor=east] {  $(X_1, \precsim_1)$} ;
\draw (3.6,4.4) node[anchor=east] {  $(X_2, \precsim_2)$} ;
\draw (9.1,2.4) node[anchor=east] {  $(X_3, \precsim_3)$} ;
\end{tikzpicture}\\
\medskip
$ $
\caption{The distributed system of three processes of Example~\ref{Equotient2}.}
\label{Fbeforequotient2}
\end{center}
\end{figure}

\begin{example}\label{Equotient2}\rm
Let $(X_1, \precsim_1), (X_2, \precsim_2)$ and $(X_3, \precsim_3)$ be three totally ordered sets (they may be uncountable, but representable in any case) such that $x_i= \max\{(X_i, \precsim_i)\} $ and $y_i= \min\{(X_i, \precsim_i)\} $, for $i=1,2,3.$
Suppose that there is a communication between these sets (defined by   $\mathcal{P}_{12},\mathcal{P}_{13}$ and $\mathcal{P}_{32}$) as it is shown in Figure~\ref{Fbeforequotient2}, such that $y_1\mathrel{\mathcal{P}}_{13} z_3, $ $x_1\mathrel{\mathcal{P}}_{12} y_2$ and $ x_3\mathrel{\mathcal{P}}_{32} y_2$.
 
If we focus on the quotient we achieve the poset of Figure~\ref{Fquotient2}. 
 Here, it is straightforward to check that $\overline{x}_1=U_{\prec_1}(y_1)\subseteq X_1$, $\overline{y}_1=L_{\precsim_1}(y_1)\subseteq X_1$, $\overline{x}_2=X_2$, 
$\overline{x}_3=U_{\precsim_3}(z_3)\subseteq X_3$ and $\overline{y}_3=L_{\prec_3}(z_3)\subseteq X_3$.

 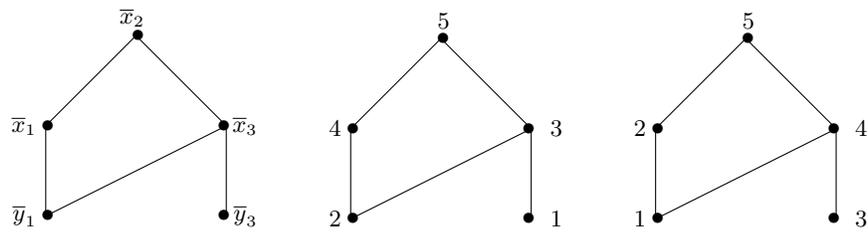
\begin{figure}[htbp]
\begin{center} 
\begin{tikzpicture}[scale=1.2] 
    \draw[] (0,2)--(1,3);
    \draw[] (2,2)--(1,3);
    \draw[] (0,1)--(0,2);
    \draw[] (2,1)--(2,2);
    \draw[] (0,1)--(2,2);
\draw (0,2) node[anchor=east] {\small $\overline{x}_1$};  
\draw (0,1) node[anchor=east] {\small $\overline{y}_1$};   
\draw (2.45,2) node[anchor=east] {\small $\overline{x}_3$};  
\draw (2.45,1) node[anchor=east] {\small $\overline{y}_3$};  
\draw (1.2,3.2) node[anchor=east] {\small $\overline{x}_2$};

\draw (0.2,2) node[anchor=east] {$\bullet$};  
\draw (0.2,1) node[anchor=east] {$\bullet$};   
\draw (2.15,2) node[anchor=east] {$\bullet$};  
\draw (2.15,1) node[anchor=east] {$\bullet$};  
\draw (1.2,3) node[anchor=east] {$\bullet$};

\end{tikzpicture}\qquad
\begin{tikzpicture}[scale=1.2] 
       \draw[] (0,2)--(1,3);
    \draw[] (2,2)--(1,3);
    \draw[] (0,1)--(0,2);
    \draw[] (2,1)--(2,2);
        \draw[] (0,1)--(2,2);
\draw (0,2) node[anchor=east] {\small $4$};  
\draw (0,1) node[anchor=east] {\small $2$};  
\draw (2.45,2) node[anchor=east] { \small $3$};  
\draw (2.45,1) node[anchor=east] {\small $1$};  
\draw (1.2,3.2) node[anchor=east] {\small $5$};

\draw (0.2,2) node[anchor=east] {$\bullet$};  
\draw (0.2,1) node[anchor=east] {$\bullet$};   
\draw (2.15,2) node[anchor=east] {$\bullet$};  
\draw (2.15,1) node[anchor=east] {$\bullet$};  
\draw (1.2,3) node[anchor=east] {$\bullet$};

\end{tikzpicture}\qquad
\begin{tikzpicture}[scale=1.2] 
       \draw[] (0,2)--(1,3);
    \draw[] (2,2)--(1,3);
    \draw[] (0,1)--(0,2);
    \draw[] (2,1)--(2,2);
        \draw[] (0,1)--(2,2);
\draw (0,2) node[anchor=east] {\small $2$};  
\draw (0,1) node[anchor=east] {\small $1$};  
\draw (2.45,2) node[anchor=east] {\small $4$};  
\draw (2.45,1) node[anchor=east] {\small $3$};  
\draw (1.2,3.2) node[anchor=east] {\small $5$};

\draw (0.2,2) node[anchor=east] {$\bullet$};  
\draw (0.2,1) node[anchor=east] {$\bullet$};   
\draw (2.15,2) node[anchor=east] {$\bullet$};  
\draw (2.15,1) node[anchor=east] {$\bullet$};  
\draw (1.2,3) node[anchor=east] {$\bullet$};
\end{tikzpicture}\\
\medskip
$ $
\caption{Hasse diagram and the corresponding bijective Richter-Peleg multi-utility $\{u_1,u_2\}$ of   the quotient associated to the distributed system of Figure~\ref{Fbeforequotient2}.}
\label{Fquotient2}
\end{center}
\end{figure}

So, now, given $w_1, w_2$ and $w_3$ three representations (that take values on $(0,1)$) of $\precsim_1$, $\precsim_2$ and $\precsim_3$, respectively, we can easily construct a representation of the distributed system through these functions and the utilities of the poset, as commented in Theorem~\ref{Prprep}:
\[v_1(x)= u_1(\bar{x})+ w_i(x), \quad \text{if } x\in X_i, \forall x\in X, \]
\[v_2(x)= u_2(\bar{x})+ w_i(x), \quad \text{if } x\in X_i, \forall x\in X.\]

It is straightforward to see that $\{v_1,v_2\}$ is also a Richter-Peleg multi-utility of the causal relation (see Proposition~\ref{Pduality}).

\end{example}

In the theorem before, the functions of the representation are defined through the sum of two functions. 
Hence, it is possible to study the continuity of the functions of the representation by means of the continuity of the other ones.

\begin{theorem}\label{Tdscont}
Let  $(\bigcup_{i=1}^n (X_i, \precsim_i),  \mathrel{\mathcal{P}}=\bigcup_{i\neq j} \mathrel{\mathcal{P}_{ij}}  )$ be a distributed system where each set $X_i$ is endowed with a topology $\tau_i$. Let $\mathcal{I}_i$ be the equivalence relation  on $X_i$ emerged from  the intersection of all the equivalence relations $\sim_{ij}^*$ and $\sim_{ji}^{**} $ (for any $ i\neq j$) on $X_i$. 
 Assume that the following conditions are satisfied for each  $i=1,...,n$:
\begin{enumerate}
\item[$(i)$] The total orders $\precsim_i$ are $\tau_i$-continuous and representable.
\item[$(ii)$] Each class $\overline{x}=\{y\in X_i\, \colon \, y\mathcal{I}_i x\}$ is open in $X_i$, for any $x\in X_i$.
\end{enumerate}

Then, the distributed system is continuously representable.
\end{theorem}
\begin{proof}
By Theorem~\ref{Prprep}, we may construct a representation of the distributed system such that each function $v_l$ is defined in $X_i$ by $v_l(x)= u_l(\bar{x})+ w_i(x)$ (as stated in Theorem~\ref{Prprep}). By hypothesis $(i)$, we may assume that $w_i$ is continuous \cite{BRME}, and by condition $(ii)$ it is straightforward to see that   $u_l$ is continuous  too. Hence,   each function $v_l$ is continuous in $X_i$, for any $i=1,...,n$.
\end{proof}

\begin{remark}
The reciprocal of the theorem above is not true 
 (see Example~11 and Remark~12 of \cite{dsjmp}).
\end{remark}

From Theorem~\ref{Tdscont} and Proposition~\ref{Pduality} the following corollary is deduced, which may be useful 
if we are focusing on a quasi-finite partially ordered set $(X,\precsim)$ endowed with a topology.

\begin{corollary}
Let $\precsim$ be a quasi-finite partial order on $(X,\tau)$. Assume that there exists a 
 distributed system $(\bigcup_{i=1}^n (X_i, \precsim_i),  \mathrel{\mathcal{P}} )$ such that the following conditions are satisfied for each  $i=1,...,n$:
\begin{enumerate}
\item[$(i)$] The total orders $\precsim_i$ are $\tau_i$-continuous and representable, where $\tau_i=\tau_{|_{X_i}}$.
\item[$(ii)$] Each class $\overline{x}=\{y\in X_i\, \colon \, y\mathcal{I}_i x\}$ is open in $X_i$, for any $x\in X_i$.
\item[$(iii)$] Any open set $U\in \tau_i$ is also open in $\tau$.
\end{enumerate}

Then, there exists a continuous and finite Richter-Peleg multi-utility of the partial order $\precsim$ on $(X,\tau)$.
\end{corollary}

 \section{Further comments} \label{sfur} 
 
For a sake of brevity and clearness, in the present paper we have argued on total orders and partial orders, however, it can be easily generalised to total preorders and preorders.

The sections related to representability (and the aggregation problem) may be implemented through partial functions, using the idea of partial representability (see \cite{partial}).
 In order to illustrate this final idea, we include the following result:
 

\begin{proposition}
Let  $(X_1,\precsim_1), (X_2,\precsim_2)$ and $(X_3,\precsim_3)$ be three representable  totally  ordered sets and $\mathcal{P}_{12}$ and $\mathcal{P}_{23}$ two communications from $X_1 $ to $X_2$ and from $X_2$ to $X_3$, respectively. Thus, the structure arised is a distributed system of three processes with line communication.  Assume that each biorder is representable, such that:
\[  
x\overline{\mathcal{P}}_{12} y \iff u_1(x)<v_1(y), \text{ for any } x\in X_1, y\in X_2,\]
\[
y\overline{\mathcal{P}}_{23} z \iff v_2(y)<w_1(z), \text{ for any }  y\in X_2, z\in X_3,
\]
as well as the biorder $\overline{\mathcal{P}}_{13}$ emerged from the composition $\overline{\mathcal{P}}_{12}\circ \overline{\mathcal{P}}_{23}$ (i.e.,  $x \overline{\mathcal{P}}_{13} z \iff x\overline{\mathcal{P}}_{12} y \overline{\mathcal{P}}_{23} z$, for some $y\in X_2$) is representable by $(u_2,w_2):$
\[  
x\overline{\mathcal{P}}_{13} z \iff u_2(x)<w_2(z), \text{ for any } x\in X_1, z\in X_3.\]

(Here, we  assume  that the functions $u_1,u_2,v_1, v_2, w_1$ and $w_2$ takes values on $(0,1)$, as well as they also represent the total  order of the corresponding set).
 Then, the associated causal relation $\rightarrow$ is partially representable (see \cite{partial}) through the functions $\{\sigma_1, \sigma_2, \sigma_3\}$ defined as follows:
\end{proposition}

 \[
  \sigma_1(x)=\left\{
  \begin{array}{ll}
   u_1(x) &\mbox{; } x\in X_1 \\[4pt]
  v_1(x) &\mbox{; } x\in X_2 \\[4pt]
   w_1(x)+1 &\mbox{; } x\in X_3 \\[4pt]
  \end{array}\right.\qquad
  \sigma_2(x)=\left\{
  \begin{array}{ll}
   u_1(x) &\mbox{; } x\in X_1 \\[4pt]
  v_2(x)+1 &\mbox{; } x\in X_2 \\[4pt]
   w_1(x)+1 &\mbox{; } x\in X_3 \\[4pt]
  \end{array}\right.\]
  
  \[
  \sigma_3(x)=\left\{
  \begin{array}{ll}
   u_2(x) &\mbox{; } x\in X_1 \\[4pt]
  \emptyset &\mbox{; } x\in X_2 \\[4pt]
   w_2(x) &\mbox{; } x\in X_3 \\[4pt]
  \end{array}\right.
\]

So that, $\quad x\rightarrow y $ if and only if $\sigma(x)<\sigma(y)$ for some $\sigma\in \{\sigma_i\}_{i=1}^3$ as well as $\sigma_i(x)<\sigma_i(y)$ for any $i=1,2,3$ such that $\sigma_i$ is defined on both $x$ and $y$.




\section*{Acknowledgments}

Asier Estevan acknowledges financial support from the Ministry of Science and Innovation of Spain under grant  
 PID2021-127799NB-I00   as well as from the UPNA  under grant JIUPNA19-2022.

\end{document}